\documentclass[11pt]{article}

\usepackage{amsmath,amsthm}

\usepackage{amssymb,latexsym}

\usepackage{enumerate}

\newtheorem{tw}{Theorem}[section]
\newtheorem{lm}[tw]{Lemma}
\newtheorem{wn}[tw]{Corollary}
\newtheorem{stw}[tw]{Proposition}
\newenvironment{dow}{\it Proof.\rm}{\hfill $\Box$}

\theoremstyle{definition}
\newtheorem{df}[tw]{Definition}
\newtheorem{uw}[tw]{Remark}

\newcommand{\BN}{{\mathbb N}}
\newcommand{\BR}{{\mathbb R}}
\newcommand{\BX}{{\mathbb X}}
\newcommand{\BM}{{\mathbb M}}

\newcommand{\FF}{{\mathcal{F}}}
\newcommand{\GG}{{\mathcal{G}}}

\newcommand{\BB}{{\mathcal{B}}}

\newcommand{\MM}{{\mathcal{M}}}

\newcommand{\EE}{{\mathcal{E}}}

\newcommand{\A}{{\mathcal{A}}}

\numberwithin{equation}{section}

\begin{document}
\title{Reduced measures for semilinear elliptic equations
involving Dirichlet operators}
\author{Tomasz Klimsiak}
\date{}
\maketitle

\begin{abstract}
We consider elliptic equations of the form (E) $-Au=f(x,u)+\mu$,
where $A$ is a negative definite self-adjoint Dirichlet operator,
$f$ is a function which is continuous and nonincreasing with
respect to $u$ and $\mu$ is a Borel measure of finite potential.
We introduce  a probabilistic definition of a solution of (E),
develop the theory of good and reduced measures introduced by H.
Brezis, M. Marcus and A.C. Ponce in the case where $A=\Delta$ and
show basic properties of solutions of (E). We also prove Kato's
type inequality. Finally, we characterize the set of good measures
in case $f(u)=-u^p$ for some $p>1$.
\end{abstract}

\footnotetext{T. Klimsiak: Institute of Mathematics, Polish
Academy of Sciences, \'Sniadeckich 8, 00-956 Warszawa, Poland, and
Faculty of Mathematics and Computer Science, Nicolaus Copernicus
University, Chopina 12/18, 87-100 Toru\'n, Poland.

e-mail: tomas@mat.umk.pl; tel.: +48 566112951; fax: +48 56 6112987.}

\section{Introduction}
\label{sec1}

Let $E$ be a separable locally compact metric space and let $m$ be
a Radon measure on $E$ such that supp$[m]=E$. In the present paper
we study semilinear equations of the form
\begin{equation}
\label{eq1.1} -Au=f(x,u)+\mu,
\end{equation}
where  $\mu$ is a Borel measure on $E$,
$f:E\times\BR\rightarrow\BR$ is a measurable function such that
$f(\cdot,u)=0$, $u\le 0$, and  $f$ is  nonincreasing and
continuous with respect to $u$. As for the operator $A$, we assume
that it is a negative definite self-adjoint Dirichlet operator on
$L^2(E;m)$. Saying that $A$ is a Dirichlet operator we mean that
\[
(Au,(u-1)^+)\le0,\quad u\in D(A).
\]
Equivalently, operator $A$ corresponds to some symmetric Dirichlet
form $(\EE,D[\EE])$ on $L^2(E;m)$ in the sense that
\begin{equation}
\label{eq1.3} D(A)\subset D[\EE],\quad \EE(u,v)=(-Au,v),\quad u\in
D(A),v\in D[\EE]
\end{equation}
(see \cite{Fukushima,MR}.)
The class of such operators  is quite large.  It contains many
local as well as nonlocal operators. The model examples are
Laplace operator $\Delta$ (or uniformly elliptic divergence form
operator) and the  fractional Laplacian $\Delta^{\alpha}$ with
$\alpha\in(0,1)$. Many other examples are to be found in
\cite{Fukushima,MR}.

Let Cap denote  the capacity  determined by $(\EE,D[\EE])$ (see
Section \ref{sec2}). It is known (see \cite{FST}) that any Borel
signed measure $\mu$ on $E$ admits a decomposition
\[
\mu=\mu_c+\mu_d
\]
into the singular (concentrated) part $\mu_c$ with respect to Cap
and  the absolutely continuous (diffuse, smooth) part $\mu_d$ with
respect to Cap. The smooth part $\mu_d$ is fully characterized in
\cite{KR:JMAA}.

The study of semilinear equations of the form (\ref{eq1.1}) in
case $\mu$ is smooth, i.e. when $\mu_c=0$,  goes back to the
papers by Brezis and Strauss \cite{BS} and Konishi \cite{Konishi}
In \cite{BS,Konishi} the existence of a solution of (\ref{eq1.1})
is proved for $\mu\in L^1(E;m)$. At present existence, uniqueness
and regularity results are available for equation (\ref{eq1.1})
involving general bounded smooth measure $\mu$  and operator
corresponding to Dirichlet form (see Klimsiak and Rozkosz
\cite{KR:JFA} for the case of symmetric regular Dirichlet form and
\cite{KR:CM} for the case of quasi-regular, possibly non-symmetric
Dirichlet form). The case $\mu_c\neq0$ is much more involved. Ph.
B\'enilan and H. Brezis \cite{BB} has observed that in such a case
equation (\ref{eq1.1}) need not have a solution even if
$A=\Delta$. In \cite{BMP} (see also \cite{BMP1}) H. Brezis, M.
Marcus and A.C. Ponce introduced the concept of good measure, i.e.
a bounded measure for which (\ref{eq1.1}) has a solution, and the
concept of reduced measure, i.e. the largest good measure, which
is less then or equal to $\mu$. In case $A=\Delta$  these concepts
are by now quite well investigated (see \cite{BB,BMP}). The
situation is entirely different in case of more general local
operators or nonlocal operators. There are known, however, some
existence and uniqueness results for (\ref{eq1.1}) in case $A$ is
a diffusion operator (see V\'eron \cite{Ve}) and in case
$A=\Delta^{\alpha}$ with $\alpha \in(0,1)$ (see Chen and V\'eron
\cite{CV}).

The main purpose of the  paper is to present a new approach to
(\ref{eq1.1}) that provides a unified way of treating
(\ref{eq1.1}) for the whole class of negative defined self-adjoint
Dirichlet operators $A$ and for $\mu$ from  some class of measures
$\BM$ including the class $\MM_b$ of bounded signed Borel measures
on $E$. In particular, we give a new definition of a solution of
(\ref{eq1.1})  and investigate the structure of good and reduced
measures relative to (\ref{eq1.1}). In case $A=\Delta$ our
definition is equivalent to the definition of a solution adopted
in \cite{BB,BMP}, so  our results generalize the results of
\cite{BB,BMP} to wide class of operators. In fact, they generalize
the existing results even in case $A=\Delta$, because in this case
$\MM_b\varsubsetneq\BM$ and $\BM$ contains important in
applications unbounded measures. The second purpose of our paper
is to give a probabilistic interpretation for solutions of
(\ref{eq1.1}).

First, some remarks concerning our definition of a solution and
the class $\BM$ are in order. Suppose we  want to consider problem
(\ref{eq1.1}) for some class of measures $\mathbb{M}$ including
$L^1(E;m)$. Considering $f\equiv 0$ in (\ref{eq1.1}) we see that
then $G:=-A^{-1}$ should be well defined on $L^1(E;m)$, i.e. the
following condition should be satisfied:
\begin{equation}
\label{eq1.03} Gg\equiv\,\uparrow
\lim_{N\rightarrow\infty}\int_0^N T_t g\,dt<\infty,\quad
m\mbox{-a.e.},\quad g\in L^{1,+}(E;m).
\end{equation}
Condition (\ref{eq1.03}) is nothing but the statement that the
semigroup $\{T_t,t\ge 0\}$ generated by $A$ (or, equivalently, the
Dirichlet form $(\EE,D[\EE]))$ is transient (see \cite[Section
1.5]{Fukushima}). It is well known that then there exists a kernel
$\{R(x,dy),x\in E\}$ such that for every $g\in L^{1,+}(E;m)$,
\[
\int_E g(y)R(\cdot,dy)=Gg,\quad m\mbox{-a.e.}
\]
If $u$ is a solution of (\ref{eq1.1}) with $f\equiv 0$ then
\[
u\cdot m(dx)=R\circ\mu(dx),
\]
where $R\circ\mu$ is a Borel measure defined as
\[
\int_E g(x) (R\circ\mu)(dx)=\int_E\int_Eg(y)R(x,dy)\mu(dx), \quad
g\in\BB^+(E).
\]
Therefore $R\circ\mu$ must be absolutely continuous with respect
to the measure $m$ for every bounded Borel measure $\mu$. This
condition is known in the literature as the Meyer hypothesis (L)
(see \cite{BG}) or the condition of absolute continuity of the
resolvent $\{G_\alpha,\alpha>0\}$ (see \cite{Fukushima}).

For the reasons explained above in the paper we assume that
$\{T_t,t\ge0\}$ is transient and hypothesis (L) is satisfied. It
is known that under these assumptions there exists a Borel
function $r:E\times E\rightarrow \BR_+$ such that
\[
r(x,y)m(dy)=R(x,dy),\quad x\in E.
\]
Using the kernel $r$ we can give our first, purely analytical
definition of a solution of (\ref{eq1.1}). Namely, we say that a
Borel function $u$ on $E$ is a solution of (\ref{eq1.1}) if
\begin{equation}
\label{eq1.2}
u(x)=\int_E f(y,u(y))r(x,y)\,dy+\int_Er(x,y)\mu(dy)
\end{equation}
for $m$-a.e. $x\in E$. Of course, to make this definition correct
we have to assume that the integrals in (\ref{eq1.2}) exist.
Therefore the class $\BM$ we consider consists of Borel measures
$\mu$ on $E$ such that  $\int_Er(x,y)\,|\mu|(dy)<\infty$ for
$m$-a.e. $x\in E$. We will show that $\MM_b(E)\subset \mathbb{M}$.
In general, the inclusion is strict. For instance, if
$A=\Delta^\alpha$, $\alpha\in(0,1]$, on an open set $D\subset
\BR^d$, then  $\BM$ includes the set of all Borel measures $\mu$
on $E$ such that $\delta^\alpha\cdot\mu\in\MM_b$, where
$\delta(x)=\mbox{dist}(x,\partial D)$. We also show  that in case
$\mu\in\MM_b$ and $A$ is a uniformly elliptic divergence form
operator on a bounded domain in $\BR^d$ definition (\ref{eq1.2})
is equivalent to  Stampacchia's definition by duality (see
\cite{Stampacchia}).

Unfortunately,  definition (\ref{eq1.2}) is rather inconvenient
for studying (\ref{eq1.1}). One of the main results of the paper
says that (\ref{eq1.2}) is equivalent to our second, probabilistic
in nature definition of a solution. At first glance the
probabilistic definition seems to be  more complicated than
(\ref{eq1.2}), but as a matter of fact suits much better to the
purposes of the present paper. Let $\mathbb{X}=(\{X_t,t\ge 0\},
\{P_x,x\in E\})$ be a Hunt process with life time $\zeta$
associated with the form $(\EE,D[\EE])$. We say that $u$ is a
probabilistic solution of (\ref{eq1.1}) if
\begin{enumerate}
\item[(a)] $f(\cdot,u)\cdot m\in\mathbb{M}$ and there exists
a local martingale additive functional $M$ of $\BX$ such that
\[
u(X_t)=u(X_0)-\int_0^tf(X_r,u(X_r))\,dr-\int_0^t\,dA^{\mu_d}_r
+\int_0^t\,dM_r,\quad t\ge 0,\quad P_x\mbox{-a.s.}
\]
for quasi every (q.e. for short)  $x\in E$ (Here $A^{\mu_d}$
denotes a continuous additive functional of $\BX$ of finite
variation in the Revuz correspondence with $\mu_d$),
\item[(b)] for every polar set $N\subset E$, every stopping time
$T\ge\zeta$ and every sequence of stopping times $\{\tau_k\}$ such
that $\tau_k\nearrow T$ and $E_x\sup_{t\le \tau_k}|u(X_t)|<\infty$
for $x\in E\setminus N$ and $k\ge 1$  we have
\[
E_xu(X_{\tau_k})\rightarrow R\mu_c(x),\quad x\in E\setminus N,
\]
where $E_x$ denotes the integration with respect to probability
$P_x$ and
\[
R\mu_c(x)=\int_E r(x,y)\mu_c(dy),\quad x\in E.
\]
\end{enumerate}
The above probabilistic definition allows us to develop a general
theory of equations of the form (\ref{eq1.1}). Moreover, in our
opinion,  the theory based on the probabilistic definition is
elegant and  simple.

We first prove  some regularity results. We show that if $u$ is a
solution of (\ref{eq1.1}) and $\mu\in\MM_b$ then $T_k(u)\in
D_e[\EE]$ and
\[
\EE(T_k(u),T_k(u))\le 2k\|\mu\|_{TV},\quad k\ge 0,
\]
where $T_k(u)=\min\{\max\{u,-k\},k\}$ and $D_e[\EE]$ is an
extension of the domain of the form $\EE$ such that the pair
$(\EE,D_e[\EE])$ is a Hilbert space (see \cite{Fukushima}). We
also prove Stampacchia's type inequality which says that for every
strictly positive excessive function $\rho$ (for $\rho\equiv 1$
for instance) and $\mu\in \MM_\rho=\{\mu\in\mathbb{M}:
\|\mu\|_{TV,\rho}:=\|\rho\cdot\mu\|_{TV}<\infty\}$,
\[
\|f(\cdot,u)\|_{L^1(E;\rho\cdot m)}\le \|\mu\|_{TV,\rho}.
\]

We next study the structure of the set $\GG$ of good measures and
the set of reduced measures relative to $A,f$. Let us recall that
the reduced measure is the largest measure $\mu^*\in\mathbb{M}$
such that $\mu^*\le\mu$ and there exists a solution of
(\ref{eq1.1}) with $\mu$ replaced by $\mu^*$. A measure
$\mu\in\BM$ is good, if $\mu^*=\mu$. By results of
\cite{KR:JFA,KR:CM}, if $\mu_c=0$, then $\mu$ is good. In the
present paper we first show that
\[
\mu-\mu^*\bot \mbox{Cap}.
\]
Then we show that, as in the case of Laplace operator, the set
$\GG$ is convex and closed under the operation of taking maximum
of two measures. We also show that $\mu\in\GG$ if and only if
\[
\mu=g-Av
\]
for some functions $g,v$ on $E$ such that $g\cdot m,
f(\cdot,v)\cdot m\in\mathbb{M}$  and $Av\in\mathbb{M}$. From this
characterization of $\GG$ we deduce that for every strictly
positive excessive function $\rho$,
\[
L^1(E;\rho\cdot m)+\mathcal{A}_\rho(f)
=\GG\cap\MM_{\rho},
\]
where
\[
\A_\rho(f)=\{\mu\in \MM_\rho:f(\cdot,R\mu)\in L^1(E; \rho\cdot m)\}.
\]
We also show that under some additional assumption on the growth
of $f$ (it is satisfied for instance if $|f(x,u)|\le c_1+c_2
e^{u^2}$), for every strictly positive excessive function $\rho$,
\[
\overline{\A_\rho(f)}=\GG\cap\MM_{\rho},
\]
where the closure is taken in the space
$(\MM_\rho,\|\cdot\|_{TV,\rho})$.

In Section \ref{sec6} we prove the so-called inverse maximum
principle and Kato's type inequality. In our context Kato's
inequality says that if $u$ is a solution of (\ref{eq1.1}) then
$Au^+\in\mathbb{M}$ and
\[
\mathbf{1}_{\{u>0\}}(Au)_d\le (Au^+)_d,\quad  (Au)^+_c= (Au^+)_c.
\]
This form of Kato's inequality for  Laplace operator was proved
by H. Brezis and A.C. Ponce in \cite{BrezisPonce}.

In the last section we study the set of good measures $\GG$  for
problem (\ref{eq1.1}) with $f$ having  at most polynomial growth,
i.e. for $f$ satisfying
\[
|f(x,u)|\le c|u|^{p},\quad x\in E,\, u\ge 0
\]
for some $p>1$. For this purpose, we introduce a new  capacity
Cap$_{A,p}$, which in the special case, when  $A=\Delta^\alpha$ on
an open bounded set $D\subset \BR^d$ with zero boundary condition
is equivalent  to the Bessel capacity
defined as
\begin{equation}
\label{eq1.5} \mbox{\rm Cap}^D_{\alpha,p}(K)
=\inf\{\|\eta\|^p_{W^{2\alpha,p}(D)}:\eta\in C_c^\infty(D),\,
\eta\ge \mathbf{1}_K\}
\end{equation}
for compact sets $K\subset D$. We prove that if $\mu\in\BM$ and
$\mu^+$ is absolutely continuous with respect to Cap$_{A,p'}$,
where  $p'$ denotes the H\"older conjugate to $p$, then a solution
of (\ref{eq1.1}) exists, i.e. $\mu\in\GG$. For $f$  of the form
\begin{equation}
\label{eq1.6} f(x,u)=-u^p,\quad x\in E, \, u\ge 0
\end{equation}
we fully characterize the set $\GG$. Namely, we prove that the
absolute continuity of $\mu^+$ with respect to Cap$_{A,p'}$ is
also necessary  for the existence of a solution of (\ref{eq1.1}).
Thus, in case $f$ is given by (\ref{eq1.6}),
\[
\GG=\{\mu\in\BM:\mu^+\ll\mbox{\rm Cap}^D_{\alpha,p'}\}.
\]
Moreover,
\[
\mu^*=\mu^+_{\mbox{\rm\tiny Cap}_{A,p'}}-\mu^-,
\]
where $\mu^+_{\mbox{\rm\tiny Cap}_{A,p'}}$ denotes the absolutely
continuous part of $\mu^+$ with respect to Cap$_{A,p'}$.

\section{Preliminaries} \label{sec2}

In the paper $E$  is a locally compact separable metric space and
$m$ is a positive Radon measure on $E$ such that supp$[m]=E$. By
$(\EE,D[\EE])$ we denote a symmetric regular Dirichlet form on
$L^2(E;m)$ (see \cite{Fukushima} or \cite{MR} for the
definitions). We will always assume that $(\EE,D[\EE])$ is
transient, i.e. there exists a strictly positive function $g$ on
$E$ such that
\[
\int_E |u(x)|g(x)\, m(dx)\le \|u\|_{\EE},\quad u\in D[\EE],
\]
where $\|u\|_\EE=\sqrt{\EE(u,u)}$, $u\in D[\EE]$. As usual, for
$\alpha>0$ we set $\EE_\alpha(u,v)=\EE(u,v)+\alpha(u,v)$, $u,v\in
D[\EE]$, where $(\cdot,\cdot)$ is the usual inner product in
$L^2(E;m)$.

By Riesz's theorem, for every $\alpha>0$  and $f\in L^2(E;m)$
there exists a unique function $G_\alpha f\in L^2(E;m)$ such that
\[
\EE_\alpha(G_{\alpha}f,g)=(f, g),\quad g\in L^2(E;m).
\]
It is an elementary check that $\{G_\alpha,\, \alpha>0\}$ is a
strongly continuous contraction resolvent on $L^2(E;m)$. By
$\{T_t,t\ge 0\}$ we denote the associated semigroup and by
$(A,D(A))$ the self-adjoint negative definite Dirichlet operator
generated by $\{T_t\}$. It is well known that $A$ satisfies
(\ref{eq1.3}) (see \cite[Section 1.3]{Fukushima}). Conversely, one
can prove (see \cite[page 39]{MR}) that for every self-adjoint
negative definite Dirichlet operator $A$ there exists a unique
Dirichlet form $(\EE,D[\EE])$ such that (\ref{eq1.3}) holds.

Given a Dirichlet form $(\EE,D[\EE])$ we define capacity
Cap$:2^E\rightarrow \mathbb{R}^{+}$ as follows: for an open
$U\subset E$ we set
\[
\mbox{Cap}(U)=\inf\{\EE(u,u):u\in D[\EE],\, u\ge
\mathbf{1}_U,\,m\mbox{-a.e.}\}
\]
and then for arbitrary $A\subset E$ we set
\[
\mbox{Cap}(A)=\inf\{\mbox{Cap}(U): A\subset U\subset E,\, U\mbox{
open}\}.
\]
An increasing sequence $\{F_n\}$ of closed subsets of $E$ is
called nest if Cap$(E\setminus F_n)\rightarrow 0$ as $n\rightarrow
\infty$. A subset $N\subset E$ is called exceptional if
Cap$(N)=0$. We say that some property $P$ holds quasi everywhere
(q.e. for short) if a set for which it does not hold is
exceptional.

We say that a function $u$  on $E$ is quasi-continuous
if there exists a nest $\{F_n\}$ such that $u_{|F_n}$ is
continuous for every $n\ge 1$. It is known that each function
$u\in D[\EE]$ has a quasi-continuous $m$-version.

A Borel measure $\mu$ on $E$ is called smooth if it does not
charge exceptional sets and there exists a nest $\{F_n\}$ such
that $|\mu|(F_n)<\infty,\, n\ge 1$. By $S$ we denote the set of
all  smooth measures on $E$.

By $S^{(0)}_0$ we denote the set of all measures $\mu\in S$ for
which there exists $c>0$ such that
\begin{equation}
\label{eq2.1} \int_E |u|\,d|\mu|\le c \sqrt{\EE(u,u)},\quad u\in
D[\EE].
\end{equation}

For a given Dirichlet form $(\EE,D[\EE])$ one can always define
the so-called extended Dirichlet space $D_e[\EE]$ as the set of
$m$-measurable functions on $E$ for which  there exists an
$\EE$-Cauchy sequence $\{u_n\}\subset D[\EE]$ convergent $m$-a.e.
to $u$ (the so-called approximating sequence). One can show that
for $u\in D_e[\EE]$ the limit $\EE(u,u)=\lim_{n\rightarrow
\infty}\EE(u_n,u_n)$ exists and does not depend on the
approximating sequence $\{u_n\}$ for $u$. Each element $u\in
D_e[\EE]$ has a quasi-continuous version. It is  known that
$(\EE,D[\EE])$ is transient if and only if $(\EE,D_e[\EE])$ is a
Hilbert space. In the latter case for a given measure $\mu\in
S_{0}^{(0)}$ inequality (\ref{eq2.1}) holds for every $u\in
D_e[\EE]$.

By $\MM_b$ we denote the set of all bounded Borel measures on $E$
and by $\MM_{0,b}$ the subset of $\MM_b$ consisting of smooth
measures.

Given a Borel measurable function $\eta$ on $E$ and a Borel
measure $\mu$ on $E$ we write
\[
( \mu, \eta )=\int_E\eta\,d\mu.
\]
By $u\cdot \mu$ w denote the Borel measure on $E$ defined as
\[
( f,u\cdot\mu)=(f\cdot u,\mu),\quad f\in \mathcal{B}(E)
\]
whenever the integrals exist.

With a regular symmetric Dirichlet form $(\EE,D[\EE])$ one can
associate uniquely  a symmetric Hunt process
$\mathbb{X}=((X_t)_{t\ge 0}, (P_x)_{x\in E},
(\FF_t)_{t\ge0},\zeta)$ (see \cite[Section 7.2]{Fukushima}). It is
related to $(\EE,D[\EE])$ by the formula
\[
T_tf(x)=E_xf(X_t),\quad t\ge 0,\quad m\mbox{-a.e.},
\]
where $E_x$ stands for the expectation with respect to the measure
$P_x$. For  $\alpha, t\ge 0$ and $f\in\mathcal{B}^{+}(E)$ we write
\[
R_\alpha f (x)=E_x\int_0^\zeta e^{-\alpha t}f(X_t)\,dt, \quad
p_tf(x)=E_xf(X_t),\quad x\in E.
\]
Observe that for $\alpha, t>0$ and $f\in L^2(E;m)$,
\[
R_\alpha f= G_\alpha f,\quad p_t f= T_t f,\quad m\mbox{-a.e.}
\]
For simplicity we denote $R_0$ by $R$. We say that some function
on $E$ is measurable if it is universally  measurable, i.e.
measurable with respect to the $\sigma$-algebra
\[
\mathcal{B}^{*}(E)=\bigcap_{\mu\in\mathcal{P}(E)}\mathcal{B}^\mu(E),
\]
where $\mathcal{P}(E)$ is the set of all probability measures on
$E$ and $\mathcal{B}^\mu(E)$ is the completion of $\mathcal{B}(E)$
with respect to the measure $\mu$.

A positive measurable function $u$ on $E$ is called
$\alpha$-excessive if for every $\beta>0$, $(\alpha+\beta)
R_{\alpha+\beta} u\le u$ and $\alpha R_\alpha u\nearrow u$ as
$\alpha\rightarrow \infty$. By $\mathcal{S}_\alpha$ we denote the
set of $\alpha$-excessive functions. We put
$\mathcal{S}=\mathcal{S}_0$.

By $S_{00}^{(0)}$ we denote the set  of  all $\mu\in S_0^{(0)}$
such that $|\mu|(E)<\infty$ and $R|\mu|$ is bounded.
For a  Borel set $B$ we set
\[
\sigma_B=\inf\{t>0; X_t\in B\},\quad D_A=\inf\{t\ge 0; X_t\in B\},
\quad \tau_B=\sigma_{E\setminus B},
\]
i.e. $\sigma_B$ is the first hitting time of $B$, $D_A$ is the
first debut time of $B$ and $\tau_B$ is the first exit time of
$B$.

By $\mathcal{T}$ we denote the set of all  stopping times with
respect to the filtration $(\FF_t)_{t\ge 0}$ and by $\mathbf{D}$
the set of all measurable functions $u$ on $E$ for which the
family
\[
\{u(X_\tau),\, \tau\in\mathcal{T}\}
\]
is uniformly integrable with respect to the measure $P_x$ for q.e.
$x\in E$.

For a Borel measure $\mu$  on $E$ and $\alpha\ge 0$ we denote by
$\mu\circ R_\alpha$  the measure defined as
\[
(f,\mu\circ R_\alpha)=(R_\alpha f,\mu),\quad f\in\mathcal{B}(E),
\]
and by $P_\mu$ we denote the measure
\[
P_\mu(A)=\int_E P_x(A)\, \mu(dx),\quad A\in \FF_\infty.
\]
In the whole paper we assume that $m$ is the  reference measure
for $\mathbb{X}$, i.e. for all $x\in E$ and $\alpha>0$ we have
$R_\alpha(x,\cdot)\ll m$. It is well known (see \cite[Lemma
4.2.4]{Fukushima}) that in this case for every $\alpha\ge 0$ there
exists a $\BB(E)\otimes\BB(E)$ measurable function
\[
r_\alpha: E\times E\rightarrow \BR^+
\]
such that for every $x\in E$ the mapping $y\mapsto r_\alpha(x,y)$
is $\alpha$-excessive and
\[
R_\alpha f (x)=\int_E f(y)r_\alpha(x,y)\,m(dy),\quad x\in E.
\]
It is also clear that by symmetry of $\mathbb{X}$,
$r_\alpha(x,y)=r_\alpha(y,x)$ for $x,y\in E,\,\alpha\ge 0$. In
what follows we put $r(x,y)=r_0(x,y),\, x,y\in E$. Thanks to the
existence of $r_\alpha$ we may define $R_\alpha\mu$ for arbitrary
positive Borel measure $\mu$ by putting
\[
R_\alpha\mu( x)=\int_E r_\alpha(x,y)\,\mu(dy).
\]
It is well known (see \cite[Section 5.1]{Fukushima} and
\cite[Theorem V.2.1]{BG} that for each $\mu\in S$ there exists a
unique perfect positive continuous additive functional $A^\mu$ in
the Revuz duality with $\mu$, and moreover,
\[
(R_\alpha \mu)(x)=E_x\int_0^\zeta e^{-\alpha t}\,dA^\mu_t,\quad
x\in E.
\]

\section{Linear equations} \label{sec3}

In this section we  give some definitions of a solution of the
linear problem
\begin{equation}
\label{eq3.1}
-Au=\mu,
\end{equation}
where $\mu$ is a Borel measure such that $R|\mu|(x)<\infty$ for
q.e. $x\in E$. The class of such measures  will be denoted by
$\BM$.

In the whole paper we adopt the convention that $\int_E
r(x,y)\,d\mu(y)=0$ for every Borel measure $\mu$ on $E$  such that
$\int_E r(x,y)\,d\mu^+(y)=\int_E r(x,y)\,d\mu^-(y)=\infty$. We
call $u:E\rightarrow\BR\cup\{-\infty,\infty\}$ a numerical
function on $E$.

\subsection{Solutions defined via the resolvent kernel
and regularity results}

\begin{df}
\label{df3.1} We say that a measurable numerical function $u$ on
$E$ is a solution of (\ref{eq3.1}) if
\[
u(x)=\int_E r(x,y)\,d\mu(y)\quad \mbox{for q.e. }x\in E.
\]
\end{df}

Let us note that  by \cite[Proposition V.1.4]{BG}, if the above equality holds
for every $x\in E$, then $u$ is Borel measurable. Since $\mu\in
\BM$, $u$ is finite q.e.

\begin{stw}
\label{stw3.1}
$\MM_b\subset\BM.$
\end{stw}
\begin{dow}
Since the form $\EE$ is assumed to be transient, there exists a
strictly positive Borel function $f$ on $E$ such that $Rf<\infty$,
q.e. From this we conclude that $f\cdot m$ is a smooth measure.
Hence, by \cite[Theorem 2.2.4]{Fukushima}, there exists an
increasing sequence $\{F_n\}$ of closed subsets of $E$ such that
$\bigcup_{n\ge 1} F_n=E$, q.e. and $\sup_{x\in E}
R(\mathbf{1}_{F_n} f)(x)<\infty$ (see also comments following
\cite[Corollary 2.2.2]{Fukushima}). As a matter of fact, in
\cite{Fukushima} in the last condition $\sup$ is replaced by
$\mbox{\rm ess}\sup$ with respect to $m$, however in view of
\cite[Proposition II.3.2]{BG}, it holds true also with supremum
norm. We have
\[
(R|\mu|,\mathbf{1}_{F_n}f)\le (|\mu|,R(\mathbf{1}_{F_n}f))
\le \|\mu\|_{TV}\cdot \|R(\mathbf{1}_{F_n}f)\|_{\infty}.
\]
Hence $R|\mu|$ is finite q.e., i.e. $\mu\in\BM$.
\end{dow}
\medskip

Using Definition \ref{df3.1}  we can easily prove some regularity
result for solutions of (\ref{eq3.1}). For this purpose, for
$k\ge0$ set
\[
T_k(u)=\max\{\min\{u,k\},-k\},\quad u\in\BR.
\]

\begin{tw}
\label{th3.1} Let $\mu\in\MM_b(E)$ and let $u$ be a solution of
\mbox{\rm{(\ref{eq3.1})}}. Then $T_k(u)\in D_e[\EE]$ and for every
$k\ge0$,
\[
\EE(T_k(u),T_k(u))\le k\|\mu\|_{TV}.
\]
\end{tw}
\begin{dow}
For  $\alpha\ge 0$ and  measurable functions $u,v$ on $E$ set
\[
\EE^{(\alpha)}(u,v)\equiv\alpha(u-\alpha R_\alpha u,v)
\]
whenever the integral exists. By the definition of a solution of
(\ref{eq3.1}), $u\cdot m=\mu\circ R$. Hence
\[
(R_\alpha u,\eta)=(u\cdot m,R_\alpha\eta)=(\mu\circ R,R_\alpha\eta)
=(\mu,RR_\alpha\eta).
\]
Therefore
\begin{align*}
\EE^{(\alpha)}(u,T_k(u))=\alpha(\mu,R T_k(u)-\alpha R_\alpha R
T_k(u))
&=\alpha (\mu, R T_k(u)-(R T_k(u)-R_\alpha T_k(u)))\\
&=(\mu,\alpha R_\alpha T_k(u))\le k \|\mu\|_{TV}.
\end{align*}
On the other hand, since $\alpha R_\alpha$ is Markovian, we have
\[
\EE^{(\alpha)}(T_k(u),T_k(u))\le \EE^{(\alpha)}(u,T_k(u)).
\]
Consequently,
\[
\sup_{\alpha\ge 0} \EE^{(\alpha)}(T_k(u),T_k(u))\le k\|\mu\|_{TV},
\]
so applying  \cite[Lemma I.2.11(ii)]{MR} we get the desired
result.
\end{dow}

\begin{uw}
(i) By Theorem \ref{th3.1}, $T_k(u)\in D[\EE]$ if $m(E)<\infty$,
because by \cite[Theorem 1.5.2(iii)]{Fukushima},
$D[\EE]=D_e[\EE]\cap L^2(E;m)$. \smallskip\\
(ii) $T_k(u)\in D[\EE]$ if the form satisfies  Poincar\'e type
inequality $c(u,u)\le \EE(u,u)$ for  every $u\in D[\EE]$ and some
$c>0$, because then $D_e[\EE]=D[\EE]$.
\end{uw}

\subsection{Probabilistic solutions}

In this subsection we give an equivalent definition  of solution
of (\ref{eq3.1}) using stochastic equations involving a Hunt
process $\mathbb{X}$ associated with the Dirichlet operator $A$.
We begin with the following lemma.

\begin{lm}
\label{lm3.1} Assume that $\mu,\nu\in\BM$ and there is
$\alpha_0\ge0$ such that $ R_\alpha \mu\ge R_\alpha\nu$ for
$\alpha\ge\alpha_0.$ Then  $\mu\ge\nu$.
\end{lm}
\begin{dow}
Since $\mu,\nu\in \BM$, there exists a strictly positive Borel
function  $\psi$ on $E$ such that $(R\psi,|\mu|+|\nu|)<\infty$.
So, it is clear that it is enough to prove that $(\eta
R\psi,\mu)\ge (\eta R\psi,\nu)$ for every $\eta\in C^+_b(E)$. Let
$\eta\in C_b^+(E)$. An elementary calculus shows that $\alpha
R_\alpha (\eta R\psi)(x)\rightarrow \eta R\psi (x)$ for every
$x\in E$. On the other hand, $\alpha R_\alpha (\eta R\psi)(x)\le
\|\eta\|_\infty R\psi(x),\, x\in E$. Hence, by the Lebesgue
dominated convergence theorem,
\[
(\eta R\psi,\mu)=\lim_{\alpha\rightarrow \infty} (\alpha R_\alpha
(\eta R\psi),\mu)= \lim_{\alpha\rightarrow \infty}(\eta
R\psi,\alpha R_\alpha\mu) \ge \lim_{\alpha\rightarrow \infty}
(\eta R\psi,\alpha R_\alpha\nu)=(\eta R\psi,\nu),
\]
which completes the proof.
\end{dow}

\begin{tw}
\label{th3.2} Assume that $\mu\in \BM^+$ and $\mu \bot$\mbox{\rm
Cap}. Then $u=R\mu$ is quasi-continuous and the process
$[0,\infty)\ni t\mapsto u(X_t) $ is a c\'adl\'ag local martingale
under the measure $P_x$ for q.e. $x\in E$.
\end{tw}
\begin{dow}
Let $u_\alpha=\alpha R_\alpha u$, $\alpha>0$. Then
\[
u_\alpha(x)=\alpha E_x\int_0^\zeta e^{-\alpha r}u(X_r)\,dr,\quad
x\in E.
\]
By the Markov property, for every $t\ge 0$ and $x\in E$ we have
\[
u_\alpha(X_t)
=\alpha E_x(\int_t^\zeta e^{-\alpha(r-t)}u(X_r)\,dr|\FF_t),
\quad P_x\mbox{-a.s.}
\]
By \cite[Theorem A.2.5]{Fukushima} the  processes $t\mapsto
u_\alpha(X_t)$, $t\mapsto u(X_t)$ are c\'adl\'ag under the measure
$P_x$ for every $x\in E$, while  by \cite[Theorem
4.6.1]{Fukushima}, $u$ is quasi-continuous. Let us put
\[
\bar{N}^{\alpha,x}_t =\alpha E_x(\int_0^\zeta e^{-\alpha
r}u(X_r)\,dr|\FF_t)-u_\alpha(X_0),\quad t\ge0,
\]
and let $N^{\alpha,x}$ denote a c\'adl\'ag modification of the
martingale $\bar N^{\alpha,x}$. Then for every $x\in E$,
\[
e^{-\alpha t} u_\alpha(X_t)=u_\alpha(X_0)
-\alpha \int_0^t e^{-\alpha r}u(X_r)\,dr+\int_0^t\,dN^{\alpha,x}_r,
\quad t\ge 0,\quad P_x\mbox{-a.s.}
\]
By the integration by parts formula applied to the processes
$e^{\alpha t}$ and $e^{-\alpha t}u_\alpha(X)$ we get
\[
u_\alpha(X_t)=u_\alpha(X_0)-\int_0^t\,dA^\alpha_r
+\int_0^t\,dM^{\alpha,x}_r,\quad t\ge 0,\quad P_x\mbox{-a.s.,}
\]
where
\[
M^{\alpha,x}_t=\int_0^t e^{-\alpha r}\, dN^{\alpha,x}_r,
\quad A^\alpha_t=\alpha\int_0^t(u-u_\alpha)(X_r)\,dr,\quad t\ge 0.
\]
Since $u$ is an excessive function,  $A^\alpha$ is an increasing
process and $u_\alpha(x)\nearrow u(x)$ for every $x\in E$ as
$\alpha\nearrow\infty$. Hence
\[
u_\alpha(X_t)\nearrow u(X_t),\quad t\ge 0,
\qquad u_\alpha(X_{t-})\nearrow u(X_{t-}),\quad t>0.
\]
Let $[u_{\alpha}(X)]$, $[u(X)]$ denote the  quadratic variations
of processes $u_{\alpha}(X)$ and $u(X)$, respectively. By
\cite[Theorem 4.2.2]{Fukushima} there exists an exceptional set
$N\subset E$ such that for every $x\in E\setminus N$,
\[
[u_\alpha(X)]_{t-}=u_\alpha(X_{t-}),\quad [u(X)]_{t-}=u(X_{t-}),
\quad t\in (0,\zeta),\quad P_x\mbox{-a.s.}
\]
Let $\zeta_i$, $\zeta_p$  denote the totally inaccessible and the
predictable  part of $\zeta$, respectively. From \cite[Theorem
4.2.2]{Fukushima} it also follows that
\[
[u_\alpha(X)]_{\zeta_i-}=u_\alpha(X_{\zeta_i-}),\quad [u(X)]_{\zeta_i-}
=u(X_{\zeta_i-}),\quad P_x\mbox{-a.s.},
\]
while by the fact that $u_\alpha, u$ are potentials,
\[
[u_\alpha(X)]_{\zeta_p-}=[u(X)]_{\zeta_p-}=0.
\]
By what has already been proved,
\[
u_\alpha(X_t)\nearrow u(X_t),\quad t\ge 0,
\qquad [u_\alpha(X)]_{t-}\nearrow [u(X)]_{t-},\quad t>0.
\]
By the generalized Dini theorem (see \cite[p.
185]{DellacherieMeyer1}), $u_\alpha(X)\nearrow u(X)$ uniformly on
compact subsets of $[0,\infty)$. Observe that for every $t\ge 0$
and q.e. $x\in E$,
\[
E_xu(X_t)\le\liminf_{\alpha\rightarrow \infty} E_xu_\alpha(X_t)
\le \liminf_{\alpha\rightarrow \infty}E_xu_\alpha(X_0)= u(x).
\]
Hence $u(X)$ is a supermartingale and $\lim_{t\rightarrow\infty}
E_xu(X_t)<\infty$. Therefore by \cite[Theorem III.13]{Protter}, for
q.e. $x\in E$ there exists  an increasing predictable process
$C^x$ with $E_xC^x_\zeta<\infty$ and  a c\'adl\'ag local
martingale $M^x$ such that
\[
u(X_t)=u(X_0)-C^x_t+M^x_t,\quad t\ge 0,\quad P_x\mbox{-a.s.}
\]
Since the filtration is quasi-left continuous, $M^x$ has no
predictable jumps. Since $X$ is quasi-left continuous, it also has
no predictable jumps, which implies that $u(X)$ has no predictable
jumps, because $u$ is quasi-continuous. Thus $C^x$ is continuous.
Since $u(X)$  is a special semimartingale,  there exists a
localizing sequence $\{\tau^x_n\}\subset \mathcal{T}$ such that
for every $n\ge 1$,
\begin{equation}
\label{eq3.2}
E_x\sup_{t\le\tau^x_n}|u(X_t)|<\infty.
\end{equation}
By \cite[Proposition 3.2]{JMP}, $\{u(X)-u_\alpha(X)\}$  satisfies
the so-called condition UT. Therefore by \cite[Corollary
2.8]{JMP}, $[u(X)-u_\alpha(X)]_t\rightarrow 0$ in probability
$P_x$ for every $t\ge 0$. But
$[u(X)-u_\alpha(X)]=[M^x-M^{\alpha,x}]$. Hence
$[M^x-M^{\alpha,x}]_t\rightarrow 0$ in probability $P_x$, which
due to (\ref{eq3.2}) is equivalent to the convergence of
$\{M^{\alpha,x}\}$ to $M^x$ in ucp (uniform on compacts in
probability). Since $u_\alpha(X)\rightarrow u(X)$ in ucp,
$A^\alpha\rightarrow C^x$ in ucp. In fact, by (\ref{eq3.2}), for
every $n\ge1$ we have
\[
E_x\sup_{t\le \tau^x_n}|A^{\alpha}_t-C^x_t|\rightarrow 0.
\]
By \cite[Lemma A.3.3]{Fukushima} there exists a process $A$ such
that $A=C^x$ for q.e. $x\in E$. Of course, $A$ is a positive
continuous additive functional. Putting
\begin{equation}
\label{eq3.03} M_t=u(X_t)-u(X_0)+A_t,\quad t\ge 0,
\end{equation}
we see that $M$ is an additive functional and $M^x=M$, $P_x$-a.s.
for q.e. $x\in E$. Thus $M$ is a local martingale additive
functional. By \cite[Theorem 5.1.4]{Fukushima} there exists
$\nu\in S$ such that $A=A^\nu$. In particular, for every
$\alpha\ge 0$,
\[
R_\alpha\nu(x)=E_x\int_0^\zeta e^{-\alpha t}\, dA^\nu_t
\]
for q.e. $x\in E$. Observe that by the resolvent identity, for
every $\alpha\ge0$ we have
\begin{equation}
\label{eq3.3} u=R_\alpha(\mu+\alpha u).
\end{equation}
On the other hand, by (\ref{eq3.2}) and the integration by parts
formula applied to the processes $e^{-\alpha t}$ and $u(X_t)$,
\begin{equation}
\label{eq3.4}
u(x)=E_x e^{-\alpha \tau^x_k}u(X_{\tau^x_k})
+E_x \int_0^{\tau^x_k}e^{-\alpha r}\,dA^\nu_r
+\alpha E_x\int_0^{\tau^x_k}e^{-\alpha r}u(X_r)\,dr.
\end{equation}
It is clear that
\[
E_x \int_0^{\tau^x_k}e^{-\alpha r}\,dA^\nu_r \rightarrow
R_\alpha\nu(x),\quad \alpha E_x\int_0^{\tau^x_k}e^{-\alpha
r}u(X_r)\,dr \rightarrow \alpha R_\alpha u(x)
\]
as $k\rightarrow\infty$. From this, (\ref{eq3.3}) and
(\ref{eq3.4}) we conclude that for q.e $x\in E$,
\[
\lim_{k\rightarrow\infty}E_x e^{-\alpha \tau^x_k}
u(X_{\tau^x_k})=R_\alpha(\mu-\nu)(x).
\]
By this and \cite[Proposition II.3.2]{BG}, $R_\alpha(\mu-\nu)\ge
0$.  Since $\alpha\ge0$ was arbitrary, applying Lemma \ref{lm3.1}
shows that $\mu\ge \nu$. Since $\mu \bot$Cap, it follows that
$\nu\equiv 0$ or, equivalently, that $A^\nu\equiv 0$. Therefore
from (\ref{eq3.03}) it follows that $u(X)$ is a local martingale.
\end{dow}
\medskip

Let us recall that a process $M$ is called a local martingale
additive functional (MAF) if it is an additive functional and $M$
is an $(\FF,P_x)$-local martingale for q.e. $x\in E$.

\begin{tw}
\label{th3.3} Assume that $\mu\in\BM^+$ and let $u=R\mu$. Then $u$
is quasi-continuous and there exists a local MAF $M$ such that
\begin{equation}
\label{eq3.5}
u(X_t)=u(X_0)-\int_0^t\,dA^{\mu_d}_r+\int_0^t\,dM_r,
\quad t\ge 0,\quad P_x\mbox{-a.s.}
\end{equation}
for q.e. $x\in E$. Moreover, for every polar set $N\subset E$,
every stopping time $T\ge\zeta$ and sequence $\{\tau_k\}\subset
\mathcal{T}$ such that $\tau_k\nearrow T$ and
$E_x\sup_{t\le\tau_k}u(X_t)<\infty$ for $x\in E\setminus N$ and
$k\ge1$ we have
\begin{equation}
\label{eq3.06} \lim_{k\rightarrow\infty}E_x
u(X_{\tau_k})=R\mu_c(x),\quad x\in E\setminus N.
\end{equation}
\end{tw}
\begin{dow}
Let $w=R\mu_c$ and $v=R\mu_d$. It is well known (see
\cite[Lemma 4.3]{KR:JFA}) that $v$ is quasi-continuous and that there
exists a uniformly integrable MAF $M^v$ such that
\begin{equation}
\label{eq3.6}
v(X_t)=v(X_0)-\int_0^t\,dA^{\mu_d}_r+\int_0^t\,dM^v_r,
\quad t\ge 0,\quad P_x\mbox{-a.s.}
\end{equation}
for q.e. $x\in E$. By Theorem \ref{th3.2}, $w$ is quasi-continuous
and there exists a local MAF $M^w$ such that
\begin{equation}
\label{eq3.7}
w(X_t)=w(X_0)+\int_0^t\,dM^w_r,\quad t\ge 0,\quad P_x\mbox{-a.s.}
\end{equation}
for q.e. $x\in E$. Let $N\subset E$ be a polar set such that
(\ref{eq3.6}),  (\ref{eq3.7}) hold for $x\in E\setminus N$. Let
$\{\tau_k\}$ be as in the formulation of the theorem. Then
$M^{v,\tau_k},\, M^{w,\tau_k}$ are both uniformly integrable and
by (\ref{eq3.6}) and (\ref{eq3.7}),
\[
u(x)=E_xu(X_{\tau_k})+E_x\int_0^{\tau_k}\,dA^{\mu_d}_r,
\quad x\in E\setminus N.
\]
Letting $k\rightarrow \infty$ in the above equation yields
\[
R\mu(x)=u(x)=\lim_{k\rightarrow \infty}E_xu(X_{\tau_k}) +
R\mu_d(x),\quad x\in E\setminus N,
\]
which proves (\ref{eq3.06}). Adding (\ref{eq3.6}) to (\ref{eq3.7})
gives (\ref{eq3.5}).
\end{dow}

\begin{uw}
\label{uw3.1} Under the assumptions of Theorem \ref{th3.3}, for
every $\alpha>0$,
\[
\lim_{k\rightarrow\infty}E_xe^{-\alpha\tau_{k}}u(X_{\tau_k})
=R_\alpha\mu_c(x),\quad x\in E\setminus N.
\]
To see this we use (\ref{eq3.3}) and arguments following it.
\end{uw}
\medskip

We are now ready to introduce the second definition of a solution
of (\ref{eq3.1}) making use of the Hunt process $\mathbb{X}$
associated with operator $A$. Solutions of (\ref{eq3.1}) in the
sense of this definition will be called probabilistic solutions or
simply solutions, because we will show that our second definition
is equivalent to the definition  via the resolvent kernel.

\begin{df}
\label{def3.9} We say that a measurable numerical function $u$ on
$E$ is a probabilistic solution of (\ref{eq3.1}) if
\begin{enumerate}
\item[(a)] there exists a local MAF $M$ such that for q.e. $x\in
E$,
\[
u(X_t)=u(X_0)-\int_0^t\,dA^{\mu_d}_r+\int_0^t\,dM_r, \quad t\ge
0,\quad P_x\mbox{-a.s.},
\]
\item[(b)] for every polar set $N\subset E$, every stopping time
$T\ge\zeta$ and  every sequence $\{\tau_k\}\subset\mathcal{T}$
such that $\tau_k\nearrow T$ and $E_x\sup_{t\le
\tau_k}|u(X_t)|<\infty$ for every $x\in E\setminus N$ and $k\ge 1$
we have
\[
E_xu(X_{\tau_k})\rightarrow R\mu_c(x),\quad x\in E\setminus N.
\]
\end{enumerate}
Any sequence $\{\tau_k\}$ with the  properties listed in (b) will
be called the reducing sequence for $u$, and we will say that
$\{\tau_k\}$ reduces $u$.
\end{df}

\begin{uw}
Since $u(X)$ in the above definition is a special semimartingale,
there exists at least one reducing sequence $\{\tau_k\}$ for $u$.
In fact, the stopping times defined as
\[
\tau_k=\inf\{t\ge 0; \, |u(X_t)|\ge k\}\wedge k,\quad k\ge1
\]
form a reducing sequence (see the reasoning in the proof of
\cite[Theorem 51.1]{Sharpe}).
\end{uw}

\begin{uw}
If $\mu$ is a smooth measure then Definition \ref{def3.9} reduces
to the definition of a solution introduced in \cite{KR:JFA}.
Indeed, by condition (a),
\[
u(x)=E_xu(X_{\tau_k})+E_x\int_0^{\tau_k}\,dA^{\mu_d}_r
\]
for q.e. $x\in E$. Therefore letting $k\rightarrow \infty$ and
using  (b) we see that for q.e. $x\in E$,
\[
u(x)=E_x\int_0^{\zeta}\,dA^{\mu_d}_r.
\]
Note that if $A$ is a uniformly elliptic divergence form operator
then by \cite[Proposition 5.3]{KR:JFA}, $u$ is also a solution of
(\ref{eq3.1}) in the sense of Stampacchia (see
\cite{Stampacchia}). In the sequel we will show that this holds
true for general Borel measures and  wider class of operators.
\end{uw}

\begin{stw}
A measurable function $u$ on $E$ is a probabilistic solution of
\mbox{\rm{(\ref{eq3.1})}} if and only if it is a solution of
\mbox{\rm{(\ref{eq3.1})}} in the sense of Definition \ref{df3.1}.
\end{stw}
\begin{dow}
Assume that $u$ is a solution of (\ref{eq3.1}) in the sense of
Definition \ref{df3.1}. Then by Theorem \ref{th3.3}, $u$ is a
probabilistic solution. Now suppose that $u$ is a probabilistic
solution of (\ref{eq3.1}). Then using (a) and (b) of the
definition of a probabilistic solution of (\ref{eq3.1}) we obtain
\[
u(x)=R\mu_c(x)+E_x\int_0^\zeta\,dA^{\mu_d}_r=R\mu(x)
=\int_Er(x,y)\,\mu(dy)
\]
for q.e. $x\in E$.
\end{dow}

\section{Semilinear equations}

In what follows  $\mu\in\BM$ and $f:E\times\BR\rightarrow \BR$ is
a function satisfying the following conditions:  $\BR\ni y\mapsto
f(x,y)$ is continuous for every $x\in E$ and $E\ni x\mapsto
f(x,y)$ is measurable for every $y\in \BR$.

In this section we  consider semilinear equation of the form
\begin{equation}
\label{eq4.1}
-Au=f(x,u)+\mu.
\end{equation}

\begin{df}
We say that a measurable numerical function $u$ on $E$ is a
solution of (\ref{eq4.1}) if $f(\cdot,u)\cdot m\in \BM$ and $u$ is
a solution of (\ref{eq3.1}) with $\mu$ replaced by
$f(\cdot,u)\cdot m+\mu$.
\end{df}

We will need the following hypotheses:
\begin{enumerate}
\item[(H1)] for every $x\in E$ the mapping $y\mapsto f(x,y)$ in nonincreasing,
\item[(H2)] for every $y\in \BR$ the mapping $x\mapsto f(x,y)\in qL^1(E;m)$,
\item[(H3)] $f(\cdot,0)\cdot m\in \BM$.
\end{enumerate}

\subsection{Comparison results, a priori estimates and regularity
of solutions}

In the sequel, for a given real function $u$ on $E$ we write

\[
f_u(x)=f(x,u(x)),\quad x\in E.
\]

\begin{stw}
\label{stw4.1} Assume that $\mu_1,\mu_2\in\BM$, $\mu_1\le\mu_2$,
$f^1(x,y)\le f^2(x,y)$ for $x\in E,\, y\in\BR$ and $f^1$ or $f^2$
satisfies \mbox{\rm{(H1)}}. Then $u_1\le u_2$ q.e., where $u_1$
(resp. $u_2$) is a solution of \mbox{\rm{(\ref{eq4.1})}} with data
$f^1,\mu^1$ (resp. $f^2,\mu^2$).
\end{stw}
\begin{dow}
Let $\{\tau_k\}$ be a common reducing sequence for $u_1$ and
$u_2$. We assume that $f^1$ satisfies (H1). By the Tanaka-Meyer
formula (see \cite[Theorem IV.66]{Protter}), for every $k\ge1$,
\begin{align*}
(u_1(x)-u_2(x))^+&\le E_x(u_1-u_2)^+(X_{\tau_k})
+E_x\int_0^{\tau_k}\mathbf{1}_{\{u_1>u_2\}}(X_r)
(f^1_{u_1}-f^2_{u_2})(X_r)\,dr\\
&\quad+E_x\int_0^{\tau_k}\mathbf{1}_{\{u_1>u_2\}}
(X_r)\,d(A^{\mu^1_d}_r-A^{\mu^2_d}_r)
\end{align*}
for q.e. $x\in E$. From the assumption $\mu_1\le\mu_2$ and
properties of the Revuz duality it follows that $dA^{\mu_d^1}\le
dA^{\mu^2_d},\, P_x$-a.s. for q.e. $x\in E$. By (H1) and the
assumptions on $f^1$ and $f^2$,
\[
\mathbf{1}_{\{u_1>u_2\}}(f^1_{u_1}-f^2_{u_2})
=\mathbf{1}_{\{u_1>u_2\}}(f^1_{u_1}-f^1_{u_2})
+\mathbf{1}_{\{u_1>u_2\}}(f^1_{u_2}-f^2_{u_2})\le 0.
\]
Hence
\[
(u_1(x)-u_2(x))^+\le E_x(u_1-u_2)^+(X_{\tau_k}),\quad k\ge 1
\]
for q.e. $x\in E$. But
\[
(u_1-u_2)^+=(R(f^1_{u_1}+\mu^1-\mu^2-f^2_{u_2}))^+
\le R(f^1_{u_1}+\mu^1-\mu^2-f^2_{u_2})^+.
\]
Therefore
\[
(u_1(x)-u_2(x))^+\le\limsup_{k\rightarrow\infty}
E_x(u_1-u_2)^+(X_{\tau_k})\le R(\mu^1_c-\mu^2_c)^+=0
\]
for q.e. $x\in E$, which proves the proposition.
\end{dow}

\begin{wn}
Under \mbox{\rm{(H1)}} there exists at most one solution of
\mbox{\rm{(\ref{eq4.1})}}.
\end{wn}

\begin{stw}
\label{stw4.2} Let $u_1,u_2$ be solutions of
\mbox{\rm{(\ref{eq4.1})}} with $\mu_1\in\BM$ and $\mu_2\in\BM$,
respectively. If $f$ satisfies \mbox{\rm{(H1)}}, then
\[
R|f_{u_1}-f_{u_2}|(x)\le R|\mu_1-\mu_2|(x),\quad x\in E.
\]
\end{stw}
\begin{dow}
Let $\{\tau_k\}$ be a common reducing sequence for $u_1$ and
$u_2$. By the Tanaka-Meyer formula,
\begin{align}
\label{eq4.2} |u_1(x)-u_2(x)|&\le
E_x|u_1-u_2|(X_{\tau_k})+E_x\int_0^{\tau_k}\mbox{sgn}(u_1-u_2)(X_r)
(f_{u_1}-f_{u_2})(X_r)\,dr \nonumber\\
&\quad+E_x\int_0^{\tau_k}\mbox{sgn}(u_1-u_2)
(X_r)\,d(A^{\mu^1_d}_r-A^{\mu^2_d}_r)
\end{align}
for q.e. $x\in E$. By (H1) the second term on the right-hand side
of (\ref{eq4.2}) is nonpositive. Therefore from (\ref{eq4.2}) it
follows that
\[
E_x\int_0^{\tau_k} |f_{u_1}-f_{u_2}|(X_r)\,dr\le
E_x|u_1-u_2|(X_{\tau_k})
+E_x\int_0^{\tau_k}\,dA^{|\mu^1_d-\mu^2_d|}_r
\]
for q.e $x\in E$. Letting $k\rightarrow \infty$ we get
\[
E_x\int_0^{\zeta} |f_{u_1}-f_{u_2}|(X_r)\,dr\le
R|\mu^1_c-\mu^2_c|(x) +R|\mu^1_d-\mu^2_d|(x)=R|\mu_1-\mu_2|(x)
\]
for q.e. $x\in E$ (see the reasoning at the end of the proof of
Proposition \ref{stw4.1}). From this and \cite[Proposition
II.3.2]{BG} we get the desired result.
\end{dow}

\begin{stw}
\label{stw4.3} Let $u$ be a solution of \mbox{\rm{(\ref{eq4.1})}}
with  $f$ satisfying \mbox{\rm{(H1), (H3)}}. Then
\[
R|f_u|(x)\le 2R|f(\cdot,0)|(x)+R|\mu|(x),\quad x\in E.
\]
\end{stw}
\begin{dow}
We apply Proposition \ref{stw4.2} to $u_1=u, u_2=0, \mu_1=\mu,
\mu_2=-f(\cdot,0)$.
\end{dow}
\medskip

Given a positive function $\rho\in \mathcal{S}$, we denote by
$\MM_\rho$ the set of all  measures $\mu\in\MM$ such that
$\|\mu\|_\rho<\infty$, where $\|\mu\|_\rho=\|\rho\cdot
\mu\|_{TV}$.

Important examples of positive $\rho\in\mathcal{S}$ are $\rho=1$
and $\rho=R\eta$, where $\eta$ is a positive  Borel function on
$E$. Let us also note that if $A=\Delta^\alpha$ (with
$\alpha\in(0,1]$) on an open bounded set $D\subset\BR^d$ (see
Remark \ref{uw.lap}) then for $\rho=R1$ we have
$\MM_\rho=\{\mu\in\MM:\delta^{\alpha}\cdot\mu\in\MM_b\}$, where
$\delta(x)=\mbox{dist}(x,\partial D)$, because by \cite{Kulczycki}
there exists $c,C>0$ such that
\[
c\delta^\alpha(x)\le R1(x)\le C\delta^\alpha(x),\quad x\in D.
\]

In the rest of the paper we assume that $\rho\in \mathcal{S}$ and
$\rho$ is strictly positive.

\begin{lm}
\label{lm4.1} Assume that $\mu,\nu\in \MM_\rho$ and $ R\mu(x)\le
R\nu(x)$ for $x\in E$. Then $\|\mu\|_{\rho}\le \|\nu\|_{\rho}$.
\end{lm}
\begin{dow}
By \cite[Proposition II.2.6]{BG}  there exists a sequence
$\{h_n\}$ of positive bounded Borel functions on $E$ such that
$Rh_n\nearrow\rho$. For  $n\ge$ we have
\[
(\mu,Rh_n)=(R\mu,h_n)\le (R\nu,h_n)=(\nu,Rh_n),
\]
so letting $n\rightarrow \infty$ we get the desired result.
\end{dow}

\begin{stw}
\label{stw4.4} Let $u_1,u_2$ be solutions of
\mbox{\rm{(\ref{eq4.1})}}  with $\mu_1\in\MM_{\rho}$ and
$\mu_2\in\MM_\rho$, respectively. If  $f$ satisfies
\mbox{\rm{(H1)}} then
\[
\|f_{u_1}-f_{u_2}\|_{L^1(E;\rho\cdot m)}\le \|\mu_1-\mu_2\|_{\rho}.
\]
\end{stw}
\begin{dow}
Follows from Proposition \ref{stw4.2} and Lemma \ref{lm4.1}.
\end{dow}

\begin{stw}
\label{stw4.5} Let $u$ be a solution of \mbox{\rm{(\ref{eq4.1})}}
with  $\mu\in\MM_\rho$ and $f$ satisfying  \mbox{\rm{(H1)}} and
such that $f(\cdot,0)\in L^1(E;\rho\cdot m)$. Then
\[
\|f_u\|_{L^1(E;\rho\cdot m)}
\le 2\|f(\cdot,0)\|_{L^1(E;\rho\cdot m)}+\|\mu\|_{\rho}.
\]
\end{stw}
\begin{dow}
Follows from Proposition \ref{stw4.3} and Lemma \ref{lm4.1}.
\end{dow}

\begin{tw}
Let $u$ be a solution of \mbox{\rm{(\ref{eq4.1})}} with
$\mu\in\MM_b$ and $f$ satisfying \mbox{\rm{(H1)}} and such that
$f(\cdot,0)\in L^1(E;m)$. Then for every $k\ge0$, $T_k(u)\in
D_e[\EE]$ and
\[
\EE(T_k(u),T_k(u))\le 2k(\|f(\cdot,0)\|_{L^1}+\|\mu\|_{TV}).
\]
\end{tw}
\begin{dow}
Follows from Theorem \ref{th3.1} and Proposition \ref{stw4.5}.
\end{dow}

\subsection{Stampacchia's definition by duality}

In \cite{Stampacchia} Stampacchia introduced a definition of a
solution of (\ref{eq3.1}) in case $\mu\in\MM_b$ and $A$ is
uniformly elliptic operator of the form
\[
A=\sum^{d}_{i,j=1}\frac{\partial}{\partial
x_j}(a_{ij}\frac{\partial}{\partial x_i})
\]
on a bounded open set $D\subset \BR^d$. According to this
definition, now called Stampacchia's definition by duality, a
measurable function $u\in L^1(D;m)$, where $m$ is the Lebesgue
measure on $\BR^d$, is a solution of (\ref{eq3.1}) if
\[
(u,\eta)=(G\eta,\mu),\quad \eta\in L^{\infty}(D;m).
\]
The above definition  has  sense, because it is well known that
for $A$ as above $G\eta$ has a bounded continuous version. In the
general case considered in the paper the original Stampacchia's
definition has to be modified, because the measure $\mu$ is not
assumed to be bounded, $G\eta$ may be not continuous for $\eta\in
L^\infty(E;m)$ and moreover, the solution of (\ref{eq3.1}) may be
not locally integrable (see \cite[Example 5.7]{KR:JFA}). In
\cite{KR:JFA} we introduced a generalized Stampacchia's definition
for solutions of (\ref{eq4.1}) with Dirichlet operator $A$ and
bounded measure $\mu$ such that $\mu\ll$Cap. Here we give a
definition for general measures of the class $\mathbb{M}$.

\begin{lm}
\label{lm.m} We have $ \mathbb{M}=\bigcup\mathcal{M}_\rho$, where
the union is taken over all strictly positive excessive bounded
functions.
\end{lm}
\begin{dow}
It is clear that  $\bigcup \mathcal{M}_\rho\subset \mathbb{M}$. To
prove the opposite inclusion, let us assume that
$\mu\in\mathbb{M}$. Then $R|\mu|<\infty$, $m$-a.e. Therefore there
exists a strictly positive Borel function $\eta$ on $E$ such that
$(R|\mu|,\eta)=(|\mu|, R\eta)<\infty$. On the other hand, since
the form $(\EE,D[\EE])$ is transient, there exists a strictly
positive Borel function $g$ on $E$ such that
$\|Rg\|_\infty<\infty$ (see \cite[Corollary 1.3.6]{Oshima}). Let us put
$\rho=R(\eta\wedge g)$. It is clear that $\rho$ is a bounded
strictly positive excessive function.
\end{dow}

\begin{df}
\label{df4.11} We say that a measurable numerical function $u$ on
$E$ is a solution of (\ref{eq4.1}) in the sense of Stampacchia if
for every $\eta\in\BB(E)$ such that $(|\mu|,R|\eta|)<\infty$
 the integrals $(u,\eta)$,
$(f_u,R\eta)$ are finite and we have
\[
(u,\eta)=(f_u,R\eta)+(\mu,R\eta).
\]
\end{df}

\begin{stw}
Let $\mu\in \mathbb{M}$. A measurable function $u$ on $E$ is a
solution of \mbox{\rm{(\ref{eq4.1})}} in the sense of Definition
\ref{df4.11} if and only if it is a solution of
\mbox{\rm{(\ref{eq4.1})}} in the sense of Definition \ref{df3.1}.
\end{stw}
\begin{dow}
Let $u$ be a solution of (\ref{eq4.1}) in the sense od Definition
\ref{df3.1}. Then by Proposition \ref{stw4.3},  $|u|+R|f_u|\le
R|\mu|$, it is clear that $u$ is a solution of (\ref{eq4.1}) in
the sense of Stampacchia. Now assume that $u$ is a solution of
(\ref{eq4.1}) in the sense of Stampacchia. By Lemma \ref{lm.m}
there exists a strictly positive $\rho\in\mathcal{S}$ such that
$\mu\in\mathcal{M_\rho}$. In fact, from the proof of Lemma
\ref{lm.m} it follows that we may take $\rho=Rg$ for some strictly
positive Borel function $g$ on $E$. We have
\[
(u,g\mathbf{1}_B)=(Rf_u,g\mathbf{1}_B)+(R\mu,g\mathbf{1}_B)
\]
for every $B\in\BB(E)$. Hence $u=Rf_u+R\mu$, $m$-a.e., and the
proof is complete.
\end{dow}

\begin{uw}
\label{uw.lap} Let $\alpha\in (0,1]$ and let $D$ be an open subset
of $\BR^d$. Denote by $(\EE,D[\EE])$ the Dirichlet form associated
with the operator $\Delta^\alpha$ on $\BR^d$ (see \cite[Example
1.4.1]{Fukushima}), and by $(\EE_D,D[\EE_D])$ the part of
$(\EE,D[\EE])$ on $D$ (see \cite[Section 4.4]{Fukushima}). By $A$
denote the operator associated with $(\EE_D,D[\EE_D])$, i.e. the
fractional Laplacian $\Delta^\alpha$ on $D$ with zero boundary
condition. If $\mu\in\MM^\alpha_\delta$ then in Definition
\ref{df4.11} one can take any function $\eta\in\BB_b(E)$ as a test
function. It follows in particular that in case of equations
involving operator $A$ Stampacchia's definition is equivalent to
the one introduced in \cite[Definition 1.1]{CV}.
\end{uw}

\begin{uw}
In \cite{KR:NDEA} renormalized solutions of (\ref{eq4.1}) are
defined in case  $\mu$ is a bounded smooth measure. It is also
proved there that $u$ is a renormalized solution of (\ref{eq4.1}) if
and only it is a probabilistic solution. Thus, in case $\mu$ is
smooth, all the definitions (renormalized, Stampacchia's by
duality, probabilistic, via the resolvent kernel) are equivalent.
\end{uw}

\begin{uw}
In case $A$ is the Laplace operator on an open bounded set
$D\subset\BR^d$, also the so-called weak solutions of
(\ref{eq4.1}) are considered in the literature (see, e.g.,
\cite{BMP}). A weak solution of (\ref{eq4.1}) is a function $u\in
L^1(D;dx)$ such that $f_u\in L^1(D;dx)$ and for every $\eta\in
C_0^\infty(\overline{D})$,
\[
-\int_D u\Delta\eta\,dx=\int_D f_u\eta\,dx+\int_D\eta\,d\mu.
\]
It is clear that the definition of weak solution  is equivalent to
Stampacchia's definition by duality. It is  worth pointing out
that in fact the  concept of weak solutions is also due to
Stampacchia (see \cite[Definition 9.1]{Stampacchia}).
\end{uw}

\subsection{Existence of solutions}

In \cite{KR:JFA} (see also \cite{KR:CM} for the case of operator
corresponding to general nonsymmetric quasi-regular form) it is
proved that if $\mu$ is smooth then under conditions (H1)--(H3)
there exists a solution of (\ref{eq4.1}). It is well known that if
$A=\Delta$ and $\mu$ is not smooth, i.e. $\mu_c\neq 0$, then in
general assumptions (H1)--(H3) are not sufficient for the
existence of a solution of (\ref{eq4.1}). In this section we give
an existence result for (\ref{eq4.1}) under the following
additional hypothesis:
\begin{enumerate}
\item[(H4)] there exists a positive Borel measurable function
$g$ on $E$ such that $g\cdot m\in\BM$ and $|f(x,y)|\le g(x)$,
$x\in E,y\in\BR$.
\end{enumerate}

Let us observe that (H4) implies (H2), (H3). In the paper we have
assumed Meyer's hypothesis (L), so we may also drop (H1).

Hypothesis (H4) imposes rather restrictive assumption on the
growth of $f$ but allows us to prove the existence of solutions
for any $\mu\in\BM$ and any Dirichlet operator $A$.

\begin{tw}
\label{tw4.2} Assume \mbox{\rm{(H4)}}. Then there exists a
solution of  \mbox{\rm{(\ref{eq4.1})}}.
\end{tw}
\begin{dow}
Let $\varrho$ be a strictly positive Borel function on $E$ such that
\[
r:=\int_E(Rg(x)+R|\mu|(x))\varrho(x)\,m(dx)<\infty.
\]
Let us define $\Phi:L^1(E;\varrho\cdot m)\rightarrow
L^1(E;\varrho\cdot m)$ by
\[
\Phi(u)=Rf(\cdot,u)+R\mu.
\]
Observe that for every $u\in L^1(E,\varrho\cdot m)$,
$\|\Phi(u)\|_{L^1(E;\varrho\cdot m)}\le r$. It is an elementary
check that $\Phi$ is continuous. Let $\{u_n\}\subset
L^1(E;\varrho\cdot m)$ and let $v_n=\Phi(u_n)$. By  \cite[Lemma
94, page 306]{DellacherieMeyer},  $\{v_n\}$ has a subsequence
convergent $m$-a.e., which when combined with the fact that
$|v_n|(x)\le Rg(x)+R|\mu|(x)$ for $x\in E$ implies that, up to a
subsequence, $\{v_n\}$ converges in $L^1(E;\varrho\cdot m)$.
Therefore by  the Schauder fixed point theorem there exists $u\in
L^1(E;\varrho\cdot m)$ such that $\Phi(u)=u$, which proves the
theorem.
\end{dow}

\section{Good measures and reduced measures}
\label{sec5}

In this section we develop the theory of reduced measures   for
(\ref{eq1.1}) in case of general Dirichlet operator $A$ and
general measure $\mu$ of the class $\BM$. Our results generalize
the corresponding results from H. Brezis, M. Marcus and A.C. Ponce
\cite{BMP} proved in the case where $A$ is the Laplace operator on
a bounded domain in $\BR^d$ and $\mu$ is a bounded measure.  Also
note that in \cite{BMP} it is assumed that $f$ does not depend on
$x$.

In the whole section in addition to (H1)--(H3) we assume that
$f(x,y)=0$ for $y\le 0$.

\begin{df}
We say that a measurable numerical function $v$ on $E$ is a
subsolution of (\ref{eq4.1}) if $f_v\cdot m\in\BM$ and there
exists a measure $\nu\in\BM$ such that $\nu\le \mu$ and
\[
-Av=f(x,v)+\nu.
\]
\end{df}

\begin{tw}
\label{tw5.1} Assume \mbox{\rm{(H1)--(H3)}}. Let $f_n=f\vee (-n)$
and let $u_n$  be a solution of
\[
-Au_n=f_n(x,u_n)+\mu.
\]
Then $u_n\searrow u^*$, where $u^*$ is a maximal subsolution of
\mbox{\rm{(\ref{eq4.1})}}. Moreover, the measure
$\mu^*=-Au^*-f(x,u^*)$ admits decomposition of the form
$\mu^*=\mu_d+\nu$ with $\nu\bot\mbox{\rm Cap}$ such that $\nu\le
\mu_c$.
\end{tw}
\begin{dow}
Let $\{\tau_k\}$ be a  reducing sequence for $u_n$. By the
Tanaka-Meyer formula,
\begin{align*}
|u_n(x)|&\le E_x|u_n|(X_{\tau_k})+E_x\int_0^{\tau_k}\mbox{sgn}(u_n)(X_r)
f_n(X_r,u_n(X_r))\,dr\\
&\quad+E_x\int_0^{\tau_k}\mbox{sgn}(u_n)(X_r)\,dA^{\mu_d}_r
\end{align*}
for q.e. $x\in E$. By (H1),
\[
|u_n(x)|+E_x\int_0^{\tau_k}|f_n(X_r,u_n(X_r))|\,dr \le
E_x|u_n|(X_{\tau_k})+E_x\int_0^{\tau_k}\,dA^{|\mu_d|}_r
\]
for q.e. $x\in E$. Letting $k\rightarrow \infty$ in the above
inequality we get
\begin{equation}
\label{eq5.2} |u_n(x)|+R|f_n(\cdot,u_n)|(x)\le R|\mu|(x)
\end{equation}
for q.e. $x\in E$. Let $v_n,w_n$ be solutions of the following
equations
\[
-Av_n=f^+_n(\cdot,u_n)+\mu^+,\qquad -Aw_n=f^-_n(\cdot,u_n)+\mu^-.
\]
Of course, $v_n,w_n$ are excessive functions and by (\ref{eq5.2}),
\begin{equation}
\label{eq5.3}
v_n=Rf^+_n(\cdot,u_n)+R\mu^+\le 2R|\mu|,
\qquad w_n=Rf^-_n(\cdot,u_n)+R\mu^-\le 2R|\mu|.
\end{equation}
By  \cite[Lemma 94, page 306]{DellacherieMeyer}, from $\{v_n\}$
and $\{w_n\}$ one can choose subsequences convergent $m$-a.e. to
excessive functions $v$ and $w$, respectively. By (\ref{eq5.3})
and \cite{GetoorGlover}, there exists $\nu_1, \nu_2\in\BM^+$ such
that $v=R\nu_1$, $w=R\nu_2$. By Theorem \ref{th3.3} the function
$h=R|\mu|$ is quasi-continuous. Therefore if we put
$\delta^1_k=\inf\{t\ge 0:h(X_t)\ge k\}\wedge\zeta$, then
$\delta^1_k\nearrow\zeta,\, P_x$-a.s. for q.e. $x\in E$. From
Theorem \ref{th3.3} it also follows that $h(X)$ is a special
semimartingale. Therefore there exists a sequence
$\{\delta^2_k\}\subset \mathcal{T}$ such that $\delta^2_k\nearrow
\zeta$ and for q.e. $x\in E$,
\[
E_x\sup_{t\le\delta_k^2}|h(X_t)|<\infty.
\]
We may assume that $\tau_k=\delta^1_k=\delta^2_k$. Since by
Proposition \ref{stw4.1}, $u_n(x)\ge u_{n+1}(x)$, $n\ge1$, for
q.e. $x\in E$, there exists $u^*$ such that $u_n\searrow u^*$,
q.e. Therefore letting $n\rightarrow\infty$ in the equation
\[
u_n(x)= E_xu_n(X_{\tau_k})+E_x\int_0^{\tau_k}
f_n(X_r,u_n(X_r))\,dr+E_x\int_0^{\tau_k}\,dA^{\mu_d}_r
\]
and using (H1)--(H3), (\ref{eq5.2}) (and the fact that
$\tau_k=\delta^1_k=\delta^2_k$) we get
\[
u^*(x)= E_xu^*(X_{\tau_k})+E_x\int_0^{\tau_k}
f(X_r,u^*(X_r))\,dr+E_x\int_0^{\tau_k}\,dA^{\mu_d}_r
\]
for q.e. $x\in E$. Observe that $u^*=v-w=R\nu$, q.e., where
$\nu=\nu_1-\nu_2$. Therefore by Theorem \ref{th3.3},
\begin{equation}
\label{eq5.4} \lim_{k\rightarrow\infty}E_xu^*(X_{\tau_k})=
R\nu_c(x)
\end{equation}
for q.e. $x\in E$. By (\ref{eq5.2}) and Fatou's lemma,
$f(\cdot,u^*)\cdot m\in\BM$. Hence $u^*$ is a solution of
(\ref{eq4.1}) with $\mu$ replaced by $\mu^*:=\mu_d+\nu_c$. What is
left is to show that  $u^*$ is the maximal subsolution of
(\ref{eq4.1}). By the construction of $u^*$, $u_n\ge u^*$.
Therefore by condition (b) of the definition of a probabilistic
solution of (\ref{eq4.1}) and Lemma \ref{lm3.1} (see also Remark
\ref{uw3.1}) we have $\mu^*_c\le \mu_c$, which when combined with
the fact that $\mu^*_d=\mu_d$ shows that $\mu^*\le \mu$, i.e. that
$u^*$ is subsolution of (\ref{eq4.1}). Suppose that $v$ is another
subsolution of (\ref{eq4.1}). Then there exists $\beta\in \BM$
such that $\beta\le\mu$ and $v$ is a solution of (\ref{eq4.1})
with $\mu$ replaced by $\beta$. Since $\beta\le \mu$ and $f_n\ge
f$, applying Proposition \ref{stw4.1} shows that $u_n\ge v$ q.e.,
hence that $u^*\ge v$ q.e., which completes the proof.
\end{dow}
\medskip

Let $\mu\in \BM$. From now on  by $\mu^*$, $u^*$ we denote the
objects constructed in Theorem \ref{tw5.1}. By Theorem
\ref{tw5.1}, $\mu^*\le\mu$. It is known (see \cite{BB}) that it
may happen that $\mu^*\neq \mu$, i.e. that there is no solution of
(\ref{eq4.1}) under assumptions (H1)--(H3).

\begin{df}
(a) We call $\mu^*$ the reduced measure associated to $\mu$.
\smallskip\\
(b) We call $\mu\in\BM$ a good measure (relative to $A$ and $f$)
if there exists a solution of (\ref{eq4.1}).
\end{df}

In what follows we denote by  $\GG$ the set of all good measures
relative to $A$ and $f$.  Of course, $\mu^*\in\GG$.

\begin{stw}
\label{stw5.1}
Let $\mu\in\BM$. Then
\begin{enumerate}
\item[\rm{(i)}] $\mu^*\le \mu$,
\item[\rm{(ii)}] $\mu-\mu^*\bot\mbox{\rm Cap}$, $(\mu^*)_d=\mu_d$,
\item[\rm{(iii)}] $\mathcal{A}\cap S\subset \GG$,
\item[\rm{(iv)}] $\mu^*$ is the largest good measure less then or equal
to $\mu$,
\item[\rm{(v)}] $|\mu^*|\le|\mu|$,
\item[\rm{(vi)}] if $\mu,\nu\in\BM$ and $\mu\le\nu$, then $\mu^*\le\nu^*$.
\end{enumerate}
\end{stw}
\begin{dow}
Assertions (i) and (ii) follow from Theorem \ref{tw5.1}. (iii)
follows from \cite{KR:JFA}. Let $\nu\in\GG$ and $\nu\le \mu$.
Since $\nu\in\GG$, there exist a solution $v$ of (\ref{eq4.1})
with $\mu$ replaced by $\nu$. Since $\nu\le\mu$, the latter means
that $v$ is a subsolution of (\ref{eq4.1}). Therefore  by Theorem
\ref{tw5.1}, $v\le u^*$ q.e. From this,  condition (b) of the
definition of a probabilistic solution and Remark \ref{uw3.1},
\[
R_\alpha\nu_c\le R_\alpha (\mu^*)_c
\]
for every $\alpha\ge 0$. Hence $\nu_c\le (\mu^*)_c$ by Lemma
\ref{lm3.1}.  On the other hand, since $\nu\le\mu$,
$\nu_d\le\mu_d$. By (ii), $(\mu^*)_d=\mu_d$. Consequently,
$\nu=\nu_c+\nu_d\le (\mu^*)_c+\mu_d=(\mu^*)_c+(\mu^*)_d=\mu^*$. To
prove  (v), let us  observe that $-\mu^-\in\GG$, because $-R\mu^-$
is a solution of (\ref{eq4.1}) with $\mu$ replaced by $-\mu^-$.
Hence, by (iv), $-\mu^-\le\mu^*$, from which we easily get (v). To
show  (vi), let us observe that $\mu^*\in\GG$, and by (i),
$\mu^*\le\nu$. Hence $\mu^*\le\nu^*$ by (iv).
\end{dow}

\begin{stw}
\label{stw5.2} A measure $\mu\in\BM$ is  good if and only if the
sequence  $\{f_n(X,u_n(X))\}$ considered in the proof of Theorem
\ref{tw5.1} is uniformly integrable under the measure $dt\otimes
P_x$ for $m$-a.e. $x\in E$.
\end{stw}
\begin{dow}
From the proof of Theorem \ref{tw5.1} we know that
$f_n(X,u_n(X))\rightarrow f(X,u^*(X))$, $dt\otimes P_x$-a.e. for
$m$-a.e. $x\in E$ and
\begin{equation}
\label{eq5.5} u_n(x)=R\mu_c+E_x\int_0^{\zeta}
f_n(X_r,u_n(X_r))\,dr+E_x\int_0^{\zeta}\,dA^{\mu_d}_r
\end{equation}
for q.e. $x\in E$. If  $\{f_n(X,u_n(X))\}$ is uniformly integrable
then letting $n\rightarrow\infty$ in (\ref{eq5.5}) shows that for
q.e. $x\in E$,
\[
u^*(x)=R\mu_c+E_x\int_0^{\zeta}
f(X_r,u^*(X_r))\,dr+E_x\int_0^{\zeta}\,dA^{\mu_d}_r,
\]
i.e.  $\mu$ is a good measure. If $\mu\in\GG$ then there exists a
solution $u$ of (\ref{eq4.1}), i.e.
\[
u(x)=R\mu_c+E_x\int_0^{\zeta}
f(X_r,u(X_r))\,dr+E_x\int_0^{\zeta}\,dA^{\mu_d}_r
\]
for q.e. $x\in E$. Of course, $u$ is a subsolution of
(\ref{eq4.1}), so by Theorem \ref{tw5.1}, $u=u^*$ and $u_n\searrow
u$. By this  and (\ref{eq5.5}),
\[
E_x\int_0^\zeta f_n(X_r,u_n(X_r))\,dr\rightarrow E_x\int_0^\zeta
f(X_r,u(X_r))\,dr
\]
for q.e. $x\in E$. Since $f_n(X,u_n(X))\rightarrow f(X,u(X))$,
$dt\otimes dP_x$-a.e. for q.e. $x\in E$ and $f_n(X,u_{n}(X))\le
0$, applying Vitali's  theorem shows that the sequence
$\{f_n(X,u_n(X))\}$ is uniformly integrable under the measure
$dt\otimes P_x$ for q.e. $x\in E$, and hence for $m$-a.e. $x\in
E$.
\end{dow}

\begin{stw}
\label{stw5.3}
If $\nu\in\BM,\, \mu\in\GG$ and $\nu\le \mu$, then $\nu\in\GG$.
\end{stw}
\begin{dow}
Let $\{u_n\}$ be the  sequence of functions of Theorem \ref{tw5.1}
associated with $\mu$ and let $\{v_n\}$ be a sequence constructed
as $\{u_n\}$ but for  $\mu$ replaced by $\nu$. By Proposition
\ref{stw4.1}, $v_n\le u_n$ q.e. Consequently,  $f_n(\cdot,u_n)\le
f(\cdot,v_n)\le 0$ q.e. Since $\mu\in\GG$, we know from
Proposition \ref{stw5.2} that the sequence $\{f_n(X,u_n(X))\}$ is
uniformly integrable under the measure $dt\otimes P_x$ for
$m$-a.e. $x\in E$. Therefore $\{f_n(X,v_n(X))\}$ has the same
property. By Proposition \ref{stw5.2}, this implies that
$\nu\in\GG$.
\end{dow}

\begin{wn}
\label{wn5.1}
If $\mu\in\BM$ and $\mu^+\in \GG$, then $\mu\in\GG$.
\end{wn}
\begin{dow}
Follows immediately from Proposition  \ref{stw5.3} and the fact
that $\mu\le\mu^+$.
\end{dow}

\begin{wn}
\label{wn5.2}
If $\mu_1,\mu_2\in\GG$, then $\mu_1\vee\mu_2\in \GG$.
\end{wn}
\begin{dow}
Let $\mu=\mu_1\vee\mu_2$. Since $\mu_1\le\mu$, $\mu_2\le \mu$ and
$\mu_1,\mu_2\in \GG$, it follows from Proposition \ref{stw5.1}(iv)
that $\mu_1\le\mu^*$ and $\mu_2\le\mu^*$. Hence $\mu\le\mu^*$. On
the other hand, by Proposition \ref{stw5.1}(i), $\mu^*\le \mu$, so
$\mu=\mu^*$, i.e. $\mu\in\GG$.
\end{dow}

\begin{wn}
\label{wn5.3}
The set $\GG$ is convex.
\end{wn}
\begin{dow}
Let $\mu_1,\mu_2\in\GG$. Then $\mu_1\vee\mu_2\in \GG$ by Corollary
\ref{wn5.2}. But for every $t\in [0,1]$, $t\mu_1+(1-t)\mu_2\le
\mu_1\vee\mu_2$, so by Proposition \ref{stw5.3},
$t\mu_1+(1-t)\mu_2\in\GG$, $t\in [0,1]$.
\end{dow}
\medskip

Set  $\GG_\rho=\GG\cap\MM_\rho$.

\begin{tw}
\label{tw5.2} We have
\begin{enumerate}
\item[\rm{(i)}]
$\|\mu-\mu^*\|_\rho=\min_{\nu\in\GG_\rho}\|\mu-\nu\|_\rho$ for
every $\mu\in\MM_\rho$,
\item[\rm{(ii)}] if $\mu_1,\mu_2\in\BM$ and $\mu_1\bot\mu_2$, then
$(\mu_1+\mu_2)^*=\mu_1^*+\mu_2^*, $
\item[\rm{(iii)}]
$(\mu\wedge\nu)^*=\mu^*\wedge\nu^*$ and
$(\mu\vee\nu)^*=\mu^*\vee\nu^*$ for every $\mu,\nu\in\BM$,
\item[\rm{(iv)}]
$(\mu^*-\nu^*)^+\le(\mu-\nu)^+ $ for every $\mu,\nu\in\BM$.
\end{enumerate}
\end{tw}
\begin{proof}
It suffices to repeat step  by step the reasoning from the proofs
of Corollary 6 and Theorems 8--10 in \cite{BMP}.
\end{proof}

\begin{tw}
\label{tw5.3} Let $\mu\in \BM$. The following conditions are
equivalent:
\begin{enumerate}
\item[\rm{(i)}] $\mu\in\GG$,
\item[\rm{(ii)}] $\mu^+\in\GG$,
\item[\rm{(iii)}] $\mu_c\in\GG$,
\item[\rm{(iv)}] $\mu=g-Av$ for some functions $g,v$ on $E$ such that
$g\cdot m\in\BM$ and $f(\cdot,v)\cdot m\in\BM$.
\end{enumerate}
\end{tw}
\begin{dow}
That (i) is equivalent to (ii) follows from Corollary \ref{wn5.1}
and Corollary \ref{wn5.2}. That (ii) implies (iii) follows from
the fact that $\mu_c\le \mu^+$ and Proposition \ref{stw5.3}.
Suppose that $\mu_c\in\GG$. Since $\mu_d\in\GG$ and
$\mu^+=\mu_d\vee\mu_c$, it follows from Corollary \ref{wn5.2} that
$\mu^+\in\GG$. Thus (iii) implies (ii). Of course (i)  implies
(iv). Suppose now that (i) is satisfied. Then
\[
-Av=f(\cdot,v)+(\mu-g-f(\cdot,v)).
\]
Hence $\mu-g-f(\cdot,v)\in\GG$, and consequently
$(\mu-g-f(\cdot,v))_c=\mu_c\in\GG$, because we already know that
(i) implies (iii). Hence $\mu\in\GG$, because we also know that
(iii) implies (i).
\end{dow}
\medskip

Set \[ L(E;m)=\{f\in\BB(E):f\cdot m\in\mathbb{M}\},
\]
\[
\mathcal{A}(f)=\{\mu\in \BM: f(\cdot,R\mu)\in L(E;m)\},\quad
\mathcal{A}_\rho=\{\mu\in \MM_\rho; f(\cdot,R\mu)\in
L^1(E;\rho\cdot m) \}.
\]

\begin{wn}
\label{wn5.4}
We have
\begin{enumerate}
\item[\rm{(i)}] $\GG+\BM\cap S\subset\GG$,
\item[\rm{(ii)}] $\mathcal{A}(f)+L(E;m)=\GG$,
\item[\rm{(iii)}] $\mathcal{A}_\rho(f)+L^1(E;\rho\cdot m)=\GG_\rho$.
\end{enumerate}
\end{wn}

Let us consider the following hypothesis:
\begin{enumerate}
\item[(A)] for every $\theta\in [0,1)$, $c\ge 0$ there exist
$\alpha(c,\theta),\,\beta(c,\theta)\ge 0$
such that
\[
|f(x,\theta u+c)|\le \alpha(c,\theta)|f(x,u)|+\beta(c,\theta),
\quad x\in E,\, u\in\BR.
\]
\end{enumerate}
\begin{tw}
\label{tw5.4} Let  $\rho\in L^1(E;m)$. If   \mbox{\rm{(A)}} is
satisfied then $\overline{\mathcal{A}_\rho(f)}=\GG_\rho$, where
the closure is taken in the space $(\MM_\rho,\|\cdot\|_{\rho})$.
\end{tw}
\begin{dow}
First we show that $\GG_\rho$ is a closed subset of
$(\MM_\rho,\|\cdot\|_\rho)$. Let $\{\mu_n\}\subset\GG_\rho$ be a
sequence such that  $\mu_n\rightarrow \mu$ in
$(\MM_\rho,\|\cdot\|_\rho)$ for some $\mu\in \MM_\rho$. Let $u_n$
denote a solution of (\ref{eq4.1}) with $\mu$ replaced by $\mu_n$
and let $\psi$ be a strictly positive Borel function on $E$ such
that $R\psi\le\rho,\, m$-a.e. Let us observe that
\begin{equation}
\label{eq5.6} (R|\mu_n-\mu_m|,\psi)\le \|\mu_n-\mu_m\|_\rho,\quad
n,m\ge1.
\end{equation}
By Proposition \ref{stw4.2},
\begin{equation}
\label{eq5.7} (R|f_{u_n}-f_{u_m}|,\psi)\le
\|\mu_n-\mu_m\|_\rho,\quad n,m\ge1.
\end{equation}
Adding (\ref{eq5.6}) to (\ref{eq5.7}) gives
\[
\int_E |u_n(x)-u_m(x)|\psi(x)\,m(dx) \le
2\|\mu_n-\mu_m\|_\rho,\quad n,m\ge1.
\]
Therefore there exists $u\in L^1(E;\psi\cdot m)$ such that
$u_n\rightarrow u$ in $L^1(E;\psi\cdot m)$. By the definition of a
solution of (\ref{eq4.1}),
\[
u_n=Rf_{u_n}+R\mu_n,\quad m\mbox{-a.e.}
\]
By (\ref{eq5.6}) and (\ref{eq5.7}) the right-hand side of the
above equation converges in  $L^1(E;\psi\cdot m)$ to
$Rf_{u}+R\mu$. Hence
\[
u=Rf_{u}+R\mu,\quad m\mbox{-a.e.},
\]
which implies that $\mu\in \GG_\rho$, and hence that $\GG_{\rho}$
is closed. Therefore
$\overline{\mathcal{A}_\rho(f)}\subset\GG_\rho$, because
$\mathcal{A}_\rho(f)\subset\GG_\rho$ by Theorem \ref{tw5.3}. Now
suppose that $\mu\in\GG_\rho$. Then there exists a solution $u$ of
(\ref{eq4.1}). Let $\theta_n=(1-\frac1n)$ and let $\{F_n\}$ be a
nest such that $c(n):=\|R(\mathbf{1}_{F_n}
f(\cdot,u))\|_\infty<\infty$ (such a nest exists, because
$f(\cdot,u)\in\mathcal{M}_\rho\subset\BM$). Let $\mu_n=-\theta_n A
u-\mathbf{1}_{F_n} f(\cdot,u)$. By (A),
\begin{equation}
\label{eq5.8} |f(x,R\mu_n(x))|\le \alpha(c(n),\theta_n)|f(x,u(x))|
+\beta(c(n),\theta_n),\quad x\in E.
\end{equation}
By Proposition \ref{stw4.5}, $f_u\in L^1(E;\rho\cdot m)$.
Therefore from  (\ref{eq5.8}) it follows that  $\mu_n\in
\mathcal{A}_\rho(f)$. Since it is clear that
$\|\mu_n-\mu\|_\rho\rightarrow 0$, we have
$\mu\in\overline{\mathcal{A}_\rho(f)}$, which completes the proof.
\end{dow}

\section{Inverse maximum principle and Kato's inequality}
\label{sec6}

In this section we consider the linear equation (\ref{eq3.1}). The
following theorem generalizes the inverse maximum principle proved
by H. Brezis and A.C. Ponce in  \cite{BrezisPonce} in case $A$ is
the Laplace operator on a bounded domain in $\BR^d$.

\begin{tw}
\label{tw6.1} Let $\mu\in \BM$ and $u$ be a solution of
\mbox{\rm{(\ref{eq3.1})}}. If $u\ge 0$ then $\mu_c\ge 0$.
\end{tw}
\begin{dow}
Assume that $u\ge 0$. Let $\{\tau_k\}$ be a  reducing sequence for
$u$. By the definition of a solution of (\ref{eq3.1}), for every
$\alpha\ge0$,
\[
\lim_{k\rightarrow\infty}E_xe^{-\alpha\tau_k}u(X_{\tau_k})=
R_\alpha\mu_c(x)
\]
for q.e. $x\in E$. In particular, $R_{\alpha}\mu_c(x)\ge0$ for
q.e. $x\in E$, and hence, by \cite[Proposition II.3.2]{BG},
$R_\alpha\mu_c\ge 0$ everywhere.  That $\mu_c\ge0$ now follows
from Lemma \ref{lm3.1}.
\end{dow}

\begin{stw}
\label{stw6.1} Assume that $\mu\in\BM$. Let $u$ be a solution of
\mbox{\rm{(\ref{eq3.1})}} and let $\varphi$ be a positive convex
Lipschitz  continuous function on $\BR$ such that $\varphi(0)=0$.
Then $A\varphi(u)\in\BM$. Moreover,
\[
\|A\varphi(u)\|_\rho\le \mbox{\rm Lip}(\varphi)\|\mu\|_\rho.
\]
\end{stw}
\begin{dow}
Let $\{\tau_k\}$ be a  reducing sequence for $u$. By the
definition  of a probabilistic solution of (\ref{eq3.1}),
\[
u(X_t)=u(X_0)-\int_0^t\,dA^{\mu_d}_r+\int_0^t\,dM_r,\quad t\ge 0
\]
for some local MAF $M$. By the It\^o-Meyer formula,
\begin{align}
\label{eq6.1}
\varphi(u)(X_t)&=\varphi(u)(X_0)-\int_0^t\varphi'(u(X_r))\,dA^{\mu_d}_r
\nonumber\\
&\quad+\int_0^t\,dA_r+\int_0^t\varphi'(u(X_{r-}))\,dM_r,\quad
t\ge0
\end{align}
for some increasing process $A$, where $\varphi'$ is the left
derivative of $\varphi$. Let $A^p$ denote the dual predictable
projection of $A$ (one can find a version of $A^p$ which is
independent of $x$; see \cite{CJPS}). Since $A^p$ is predictable,
it is continuous, because the filtration $(\FF_t)$ is quasi-left
continuous. Therefore there exists a positive smooth measure $\nu$
such that $A^p=A^\nu$. For q.e. $x\in E$ we have
\begin{align*}
E_x\int_0^\zeta\,dA^\nu_r&=\lim_{k\rightarrow \infty}
E_x\int_0^{\tau_k}\,dA^\nu_r\le \lim_{k\rightarrow
\infty}(E_x\varphi(u(X_{\tau_k}))
+E_x\int_0^{\tau_k}\varphi'(u(X_r))\,dA^{\mu_d}_r)\\& \le
\mbox{\rm Lip}(\varphi)\lim_{k\rightarrow
\infty}(E_x|u(X_{\tau_k})|+E_x\int_0^{\tau_k}\,dA^{|\mu_d|}_r) \le
2\mbox{\rm Lip}(\varphi)R|\mu |(x).
\end{align*}
Thus $\nu\in\BM$. Write
\[
v_1(x)=R\nu(x),\quad v_2(x)=R\mu_d^-(x),\quad x\in E
\]
and observe that
\[
\varphi(u)(X_t)+v_1(X_t)+v_2(X_t)
=\varphi(u)(x)+v_1(x)+v_2(x)-\int_0^t\,dA^{\mu_d^+}
+\int_0^t\,d\bar{M}_r,\quad t\ge 0
\]
for some local MAF $\bar{M}$. Set $w=\varphi(u)+v_1+v_2$. From the
above equation and the fact that $w\ge0$ it follows that $w(X)$ is
a supermartingale. Therefore $w$ is an excessive function. On the
other hand,
\[
w\le |\varphi(u)|+v_1+v_2\le \mbox{\rm Lip}
(\varphi)|u|+R\nu+R\mu_d^-\le {\rm Lip}
(\varphi)R|\mu|+R\nu+R\mu_d.
\]
Therefore by \cite[Proposition 3.9]{GetoorGlover} there exists a
positive $\beta\in\BM$ such that $w=R\beta$. This implies that
$A\varphi(u)=\beta-\nu-\mu_d^-\in\BM$. By (\ref{eq6.1}) and the
assumptions on $\varphi$,
\begin{align*}
\varphi(u)(x)&=E_x\varphi(u)(X_{\tau_k})
+E_x\int_0^{\tau_k}\varphi'(u(X_r))\,dA^{\mu_d}_r
-E_x\int_0^{\tau_k}\,dA^\nu_r \\
& \le\mbox{\rm Lip}(\varphi)
(E_x|u|(X_{\tau_k})+E_x\int_0^{\tau_k}\,dA^{|\mu_d|}_r)
\end{align*}
for q.e. $x\in E$. Letting $k\rightarrow\infty$ and applying Lemma
\ref{lm4.1} we get the desired result.
\end{dow}
\medskip

The following version of Kato's inequality was proved by H. Brezis
and A.C. Ponce \cite{BrezisPonce} (see also  H. Brezis, M. Marcus
and A.C. Ponce \cite{BMP}) in case $A$ is the Laplace operator on
a bounded domain in $\BR^d$).

\begin{tw}
Let $u$ be a solution of \mbox{\rm{(\ref{eq3.1})}}. Then
$Au^+\in\BM$ and
\begin{equation}
\label{eq6.2}
\mathbf{1}_{\{u>0\}}(Au)_d\le (Au^+)_d,
\end{equation}
\begin{equation}
\label{eq6.3}
(Au)^+_c= (Au^+)_c.
\end{equation}
\end{tw}
\begin{dow}
By Proposition \ref{stw6.1} and (\ref{eq6.1}), $Au^+\in\BM$ and
there exist positive $\nu,l\in\BM$ such that $\nu\bot$Cap,
$l\ll$Cap and
\[
-Au^+=\nu+\mathbf{1}_{\{u>0\}}\mu_d-l.
\]
By the resolvent identity,  for every $\alpha\ge 0$ we have
\[
u=R_\alpha(\mu+\alpha u),\quad
u^+=R_\alpha(\nu+\mathbf{1}_{\{u>0\}}\mu_d-l+\alpha u^+).
\]
It is clear that
\[
R_\alpha(\nu+\mathbf{1}_{\{u>0\}}\mu_d-l+\alpha u^+)\le
R_\alpha(\mu+\alpha u)^+.
\]
Hence
\[
R_\alpha(\nu+\mathbf{1}_{\{u>0\}}\mu_d-l)\le R_\alpha[(\mu+\alpha
u)^+-\alpha u^+]\le R_\alpha\mu^+.
\]
By Lemma \ref{lm3.1},
\[
\nu+\mathbf{1}_{\{u>0\}}\mu_d-l\le \mu^+.
\]
Taking the diffuse part of the above inequality we get
(\ref{eq6.2}). Taking the concentrated part we get
\begin{equation}
\label{eq6.4} \nu\le\mu^+_c.
\end{equation}
On the other hand, since $u^+-u\ge 0$, it follows from Theorem
\ref{tw6.1} that
\begin{equation}
\label{eq6.5}
(\nu+\mathbf{1}_{\{u>0\}}\mu_d-l-\mu)_c\ge 0,
\end{equation}
which implies that $\nu\ge\mu_c^+$. When combined with
(\ref{eq6.4}) this gives (\ref{eq6.2}).
\end{dow}

\begin{uw}
Applying in the proof of Theorem \ref{tw6.1} the It\^o-Meyer
formula with right derivative of the function $u\mapsto u^+$ we
obtain  (\ref{eq6.5}) with $\mathbf{1}_{\{u>0\}}$ replaced by
$\mathbf{1}_{\{u\ge0\}}$. As a result, we get (\ref{eq6.2}) with
$\mathbf{1}_{\{u>0\}}$ replaced by $\mathbf{1}_{\{u\ge0\}}$.
\end{uw}

\section{Equations with polynomial nonlinearity}

In this section we give a necessary and sufficient  condition on
$\mu$ ensuring  the existence of a solution of (\ref{eq4.1}) with
$f$ satisfying the condition
\begin{equation}
\label{eq7.1} |f(x,u)|\le cu^p,\quad x\in E,u\ge0
\end{equation}
for some constants $c\ge0$, $p>1$. We also calculate the reduced
measure in the case where $f(x,u)=-u^p$. In our study a primary
role will be played by a new capacity $\mbox{Cap}_{A,p}$, which we
define below.

Let $p\ge1$. By the Riesz-Thorin interpolation theorem one can
extend the semigroup $\{T_t,t\ge 0\}$ from $L^2(E;m)\cap L^p(E;m)$
to $L^p(E;m)$. We denote the extended semigroup by $\{T^p_t,t\ge
0\}$, whereas  by $\{R^p_\alpha,\alpha>0\}$ we denote its
resolvent. Let $(A_p, D(A_p))$ be the operator generated by
$\{T^p\}$. It is well known that $D(A_p)=R^p_1(L^p(E;m))$. We set
$D_+(A_p)=R^p_1(L^{p,+}(E;m))$. Each element of $D_+(A_p)$ is
defined pointwise via the resolvent kernel. Let $V_p$ denote the
space $D(A_p)$ equipped with the norm
\[
\|u\|_{V_p}=\|A_pu\|_{L^p(E;m)}+\|u\|_{L^p(E;m)}.
\]
We  define the  capacity of $B\subset E$ as
\[
\mbox{Cap}_{A,p}(B)=\inf\{\|\eta\|_{V_p}^p:\eta\in D_+(A_p),\,
\eta\ge \mathbf{1}_{B}\}.
\]
It is an elementary check that Cap$_{A,p}$ is subadditive and
increasing (see, e.g., \cite[Proposition 2.3.6]{AdamsHedberg}). We
say that $\mu\in V'_p\cap\mathbb{M}^+$ if  for every $\eta\in
V^+_p$,
\[
(\eta,\mu)\le c\|\eta\|_{V_p}.
\]

In the rest of the section we assume that $p>1$. By $p'$ we denote
the H\"older conjugate to $p$.
\begin{stw}
\label{stw7.1} If $\mu\in V'_p \cap\mathbb{M}^+$ then $\mu$ is a
good measure relative to the function $f(u)=-|u|^{p'}$.
\end{stw}
\begin{dow}
Let $u$ be a solution of the equation
\[
(I-A)u=\mu.
\]
Then $u\in L^{p'}(E;m)\cap L(E;m)$. Indeed, the fact that $u\in
L(E;m)$ follows from the inequality $Ru\le R\mu$. Now, for  $f\in
L^{p,+}(E;m)$ set $\eta=R_1^p f$. Then
\begin{align*}
\int_E u f\,dm &=\int_E u (I-A_p)\eta\, dm=\int_E (I-A_p)u\eta\,dm
=\int_E\eta\,d\mu\\
&\le c\|\eta\|_{V_p}=c(\|A_p\eta\|_{L^p(E;m)}
+\|\eta\|_{L^p(E;m)})\le 2c\|f\|_{L^p(E;m)},
\end{align*}
which shows that $u\in L^{p'}(E;m)$. That $\mu$ is a good measure
relative to $f(u)=-|u|^{p'}$ now follows from Theorem \ref{tw5.3}.
\end{dow}

\begin{lm}
\label{lm7.1} Let $u\in D_+(A_p)$. Then for every $\lambda>0$,
\[
\mbox{\rm Cap}_{A,p}(u\ge\lambda)\le \lambda^{-p}\|u\|_{V_p}^p.
\]
\end{lm}
\begin{dow}
Let $B=\{u\ge\lambda\}$. Then $\lambda^{-1}u\ge \mathbf{1}_B$, so
the required inequality follows immediately from the definition of
$\mbox{\rm Cap}_{A,p}$.
\end{dow}

\begin{lm}
\label{lm7.2} Let $\mu\in\MM_b^+$. If $\mu\le c\cdot$\mbox{\rm
Cap}$_{A,p}$ for some $c\ge 0$, then $\mu\in V_{p}'$.
\end{lm}
\begin{dow}
Let $\eta\in V^+_p$. By our assumptions on $\mu$ and Lemma
\ref{lm7.1}, for any $\eta\in V^+_p$ with $\|\eta\|_{V_p}=1$ we
have
\begin{align*}
\int_E\eta\, d\mu\le \mu(E)+\sum_{k=0}^{\infty} 2^{k+1}\mu(\eta\ge
2^k)&\le \mu(E)+c\sum_{k=0}^{\infty} 2^{k+1}\mbox{\rm
Cap}_{A,p}(\eta\ge 2^k)\\
&\le\mu(E)+c\sum_{k=0}^{\infty}2^{k(1-p)+1}<\infty,
\end{align*}
which proves the lemma.
\end{dow}
\medskip


\begin{lm}
\label{lm7.3} Let $\mu\in\MM_b^+$ and $\mu\ll\mbox{\rm
Cap}_{A,p}$. Then there exists a decreasing sequence $\{G_n\}$ of
Borel subsets of $E$ such that
\[
\lim_{n\rightarrow\infty}\mbox{\rm Cap}_{A,p}(G_n)= 0,\quad
\lim_{n\rightarrow\infty}\mu(G_n)=0, \quad \mathbf{1}_{E\setminus
G_n}\cdot \mu\le 2^n \mbox{\rm Cap}_{A,p}\,,\quad n\ge 1.
\]
\end{lm}
\begin{dow}
It is enough to repeat step by step the proof of \cite[Lemma
2.2.9]{Fukushima}, the only difference being in the fact that we
choose the sets $B_n$ appearing in the proof of \cite[Lemma
2.2.9]{Fukushima} as Borel sets.
\end{dow}
\medskip

As a corollary to Lemma \ref{lm7.3} we get the following
proposition.

\begin{stw}
\label{stw7.2} A measure $\mu\in\mathbb{M}^+$ satisfies
$\mu\ll\mbox{\rm Cap}_{A,p}$ if and only if there exists an
increasing sequence $\{E_n\}$ of Borel subsets of $E$ such that
$\mathbf{1}_{E_n}\cdot \mu\in V'_p\cap \mathbb{M}^+$ for $n\in\BN$
and $\mu(E\setminus \bigcup_{n\ge 1}E_n)=0$.
\end{stw}

\begin{tw}
\label{tw7.1} Assume \mbox{\rm(\ref{eq7.1})}. If
$\mu\in\mathbb{M}$ and $\mu^+\ll\mbox{\rm Cap}_{A,p'}$ then
$\mu\in\GG$.
\end{tw}
\begin{dow}
By Theorem \ref{tw5.3} we may assume that $\mu\ge 0$.  By Lemma
\ref{lm.m} there exists a strictly positive bounded excessive
function $\rho$ such that $\mu\in\MM_\rho^+$, and by Proposition
\ref{stw7.2} there exists a sequence $\{\mu_n\}\subset V'_{p'}\cap
\mathbb{M}^+$ such that
$\lim_{n\rightarrow\infty}\|\mu_n-\mu\|_\rho=0$. Therefore it is
enough to show that $\mu_n\in\GG$. But this follows from
Proposition \ref{stw7.1}.
\end{dow}

\begin{wn}
\label{wn7.6} Assume that $\mu\in\BM$ and an let $ f(x,u)=-u^p$,
$x\in E$, $u\ge0$. Then $\mu\in\GG$ if and only if
$\mu^+\ll\mbox{\rm Cap}_{A,p'}$.
\end{wn}
\begin{dow}
Sufficiency follows from Theorem \ref{tw7.1}. Suppose that $\mu\in
\GG$. By Theorem \ref{tw5.3}, $\mu^+\in\GG$. By Proposition
\ref{stw5.3} and closedness of $\GG$ we may assume that $\mu^+$ is
bounded. Assume that Cap$_{A,p'}(B)=0$ for some Borel set
$B\subset E$. Then there exists a sequence $\{\eta_n\}\subset
V^+_{p'}$ such that $\|\eta_n\|_{V_{p'}}\rightarrow 0$,
$\sup_{n\ge1}\eta_n\le c$ for some $c>0$ and $\eta_n\ge
\mathbf{1}_{B}$. Let $u$ be a solution of (\ref{eq4.1}) with $\mu$
replaced by $\mu^+$. Then $u\in L^p(E;m)$ by Proposition
\ref{stw4.5}. Therefore
\begin{align*}
\mu^+(B)&\le (\eta_n,\mu^+)=(u^p,\eta_n)+(u,-A_p\eta_n)\\& \le
(u^p,\eta_n) +\|u\|_{L^p(E;m)} \|A_p\eta_n\|_{L^{p'}(E;m)}\le
(u^p,\eta_n)+\|u\|_{L^p(E;m)} \|\eta_n\|_{V_{p'}}
\end{align*}
for every $n\in\BN$, which forces $\mu^+(B)=0$.
\end{dow}

\begin{wn}
Let the assumptions of Corollary \ref{wn7.6} hold. Let
$\mu^+_{\mbox{\rm\tiny Cap}_{A,p'}}$ denote the absolutely
continuous part, with respect to \mbox{\rm Cap}$_{A,p'}$,  of the
measure $\mu^+$. Then
\[
\mu^*=\mu^+_{\mbox{\rm\tiny Cap}_{A,p'}}-\mu^-.
\]
\end{wn}
\begin{dow}
It suffices to repeat step by step the proof of \cite[Theorem
16]{BMP}.
\end{dow}
\bigskip
\begin{uw}
\label{uw.bes} Let us note that from \cite[Proposition
2.3.13]{AdamsHedberg} (see also \cite{Gurarie}) it follows that
for all $p>1$, $\alpha\in (0,1]$ and open bounded set
$D\subset\BR^d$,
\[
c_1\mbox{\rm Cap}^D_{\alpha,p}(B)\le\mbox{\rm Cap}_{A,p}(B)\le c_2
\mbox{\rm Cap}^D_{\alpha,p}(B), \quad B\subset D,
\]
where $A=\Delta^\alpha$ on $D$ with zero boundary  condition (see
Remark \ref{uw.lap}) and for a compact $K\subset D$ the capacity
$\mbox{\rm Cap}^D_{\alpha,p}(K)$ is defined by (\ref{eq1.5}).
\end{uw}

\noindent{\bf\large Acknowledgements}
\medskip\\
Research supported by National Science Centre Grant No.
2012-07-D-ST1-02107.

\end{document}